\documentclass[a4,reqno,11pt]{amsart}
\usepackage{amsmath, amsfonts, amssymb, amsthm, amscd}
\usepackage[pdftex]{graphicx}
\usepackage[pdftex,dvipsnames]{xcolor}
\usepackage{mathtools}
\usepackage[percent]{overpic}

\usepackage[colorlinks]{hyperref}
\DeclareGraphicsExtensions{.ps,.eps}

\usepackage[latin1]{inputenc}
\usepackage{bbm}

\numberwithin{equation}{section}

\newtheorem{theorem}{Theorem}[section]

\newtheorem{lemma}{Lemma}[section]
\newtheorem{proposition}{Proposition}[section]

\newtheorem{remark}{Remark}

\newcommand{\blue}{\textcolor{blue}}

\newcommand{\abs}[1]{\left| #1\right|}

\newcommand{\calA}{\mathcal{A}}

\newcommand{\calL}{\mathcal{L}}

\newcommand{\calP}{\mathcal{P}}

\newcommand{\calZ}{\mathcal{Z}}

\newcommand{\bbA}{\mathbb{A}}

\newcommand{\bbE}{\mathbb{E}}

\newcommand{\bbL}{\mathbb{L}}

\newcommand{\bbN}{\mathbb{N}}

\newcommand{\bbP}{\mathbb{P}}

\newcommand{\bbR}{\mathbb{R}}

\newcommand{\bbZ}{\mathbb{Z}}

\newcommand{\sfg}{\mathsf g}

\newcommand{\ua}{\underline{a}}

\newcommand{\uh}{\underline{h}}
\newcommand{\ur}{\underline{r}}

\newcommand{\uv}{\underline{v}}
\newcommand{\uu}{\underline{u}}
\newcommand{\ux}{\underline{x}}
\newcommand{\uy}{\underline{y}}

\newcommand{\ueta}{\underline{\eta}}
\newcommand{\unu}{\underline{\nu}}
\newcommand{\urho}{\underline{\rho}}
\newcommand{\ukappa}{\underline{\kappa}}
\newcommand{\uxi}{\underline{\xi}}
\newcommand{\uzeta}{\underline{\zeta}}
\newcommand{\uchi}{\underline{\chi}}
\newcommand{\udelta}{\underline{\delta}}

\newcommand{\uX}{\underline{X}}
\newcommand{\uY}{\underline{Y}}

\newcommand{\lb}{\left(}
\newcommand{\rb}{\right)}
\newcommand{\lbr}{\left\{}
\newcommand{\rbr}{\right\}}

\newcommand{\dd}{{\rm d}}

\newcommand{\1}{\mathbbm{1}}

\newcommand{\df}{\stackrel{\Delta}{=}}

\newcommand{\be}[1]{\begin{equation}\label{#1}}
\newcommand{\ee}{\end{equation}}

\newcommand{\fkg}{\stackrel{\mathsf{FKG}}{\prec}}

\definecolor{darkblue}{rgb}{0,0.3,0.9}

\newcommand{\uell}{\underline{\ell}}

\begin{document}
\date{\today} 

\title[Polymers under Geometric Area Tilts ]
{Confinement of Brownian Polymers under Geometric Area Tilts}

\author{Pietro Caputo}
\address{Dipartimento di Matematica e Fisica, Roma Tre University, Rome, Italy}
\email{caputo@mat.uniroma3.it}

\author{Dmitry Ioffe}
\address{Faculty of IE\&M, Technion, Haifa 32000, Israel}
\email{ieioffe@ie.technion.ac.il}
\thanks{DI was supported by the Israeli Science Foundation grant 
1723/14.}

\author{Vitali Wachtel}
\address{Institut f\"ur Mathematik, Universit\"at Augsburg, D-86135 Augsburg, 
Germany}
\email{vitali.wachtel@math.uni-augsburg.de}
\thanks{VW was supported by the Humboldt Foundation.}

\maketitle

\begin{abstract}
We consider tightness for families of non-colliding Brownian bridges  above a hard wall,
which 
are subject to geometrically growing self-potentials of tilted area type. 
The model is introduced in order to mimic level lines of $2+1$  discrete Solid-On-Solid 
 random interfaces above a hard wall. 
\end{abstract}

\section{Brownian polymers under 
geometric area tilts. } 
\subsection{Introduction}
Ensembles of non-intersecting random lines, 
 both in the discrete and continuous setups, as well as their  scaling limits as the 
 linear size of the system
 grows,   
play a significant role in  the probabilistic analysis of
various problems in random matrices, interacting particle systems and 
effective interface models; see e.g.\ \cite{FS2003step,johansson2005random,spohn2005kardar,
corwinhammond,BCR2015,Bornemann, 
johansson2017edge,WFS2017,CD2018,duits2018} 
and references therein. 

A salient feature of these ensembles is that  the   unconditional reference distribution of the individual 
random lines is assumed to be the same for all the random lines in a stack. 
The ensuing exchangeability paves the way for an application of Karlin-McGregor type formulas, and gives 
rise  to various determinantal structures.

In this paper, we introduce a  
model consisting of an unbounded number of non-intersecting 
Brownian bridges,  
above a hard wall and
subject to geometrically increasing area tilts. 

	Thus, in our model reference statistics
of individual  lines depends on their serial numbers  - there is a stronger pressure towards the wall on 
paths further down the stack.  The main result of the paper is that the ensemble in question does not blow up
as the number of bridges grows to infinity. There is a longer program to attend to: In subsequent works we shall
describe the limiting (infinite) line ensemble and try to derive appropriate scaling limits from tilted random walks
and, eventually, from level lines of discrete random interfaces. Indeed, 
the main motivation comes from the study of the fluctuations of level lines in the random surface separating 
low-temperature phases in the solid-on-solid (SOS) approximation of the 3D Ising model. 
Before describing our model and main results, 
let us give more details about the context where the model naturally arises, 
we refer to \cite{caputoetal2014,caputoetal2016,IV2016,lacoin2017wetting},
for further  information. 


\subsection{Level lines of SOS interfaces above a hard wall.} 
Given a large integer $L$, consider the square box $\Lambda_L$ of side $2L$ in $\bbZ^2$, centered at the origin.  The $(2+1)$-dimensional SOS Gibbs measure on $\Lambda_L$ with zero boundary conditions is the distribution $\mu_L$ over integer height functions $\varphi:\Lambda_L\mapsto \bbZ$, such that the probability of each  given configuration $\varphi$ is proportional to  
$$
\exp\left(-\beta\sum_{x\sim y}|\varphi(x)-\varphi(y)| \right),
$$
where $\beta>0$ denotes the inverse temperature, and the sum extends over all pairs of 
neighbouring sites in $\bbZ^2$, with the boundary constraint $\varphi(y)=0$ for all 
$y\in\bbZ^2\setminus \Lambda_L$.   When $\beta$ is larger than some fixed constant, which we 
assume throughout this discussion, the surface $\varphi$ with distribution $\mu_L$ is typically 
flat around height zero with local upward and downward fluctuations of depth $h$ with density roughly 
proportional to $e^{-4\beta h}$ for all $h\in\bbN$. 
If the surface is constrained to stay above a hard wall, that is $\mu_L$ is conditioned on the 
event $\{\varphi (x)\geq 0, x\in \Lambda_L\}$, then it is well known that the wall pushes 
the surface globally to a height $$H(L)\sim \frac1{4\beta}\log L,$$ see 
\cite{bricmontetal1986,caputoetal2014}. This phenomenon, known as entropic repulsion, is 
heuristically explained as follows: a global shift of the surface from height $h-1$ to height $h$ provides 
room for downward fluctuations of depth $h$, which gives a bulk entropic gain of order $|\Lambda_L |e^{-4\beta h}$, 
while forcing an energy loss proportional to the size of the boundary $|\partial\Lambda_L|$; the surface then 
stabilises when energy and entropy balance out, that is when $h$ equals  $H(L)\sim \frac1{4\beta}\log L$. 

\begin{figure}
\includegraphics[angle =0,height = 4.5cm]{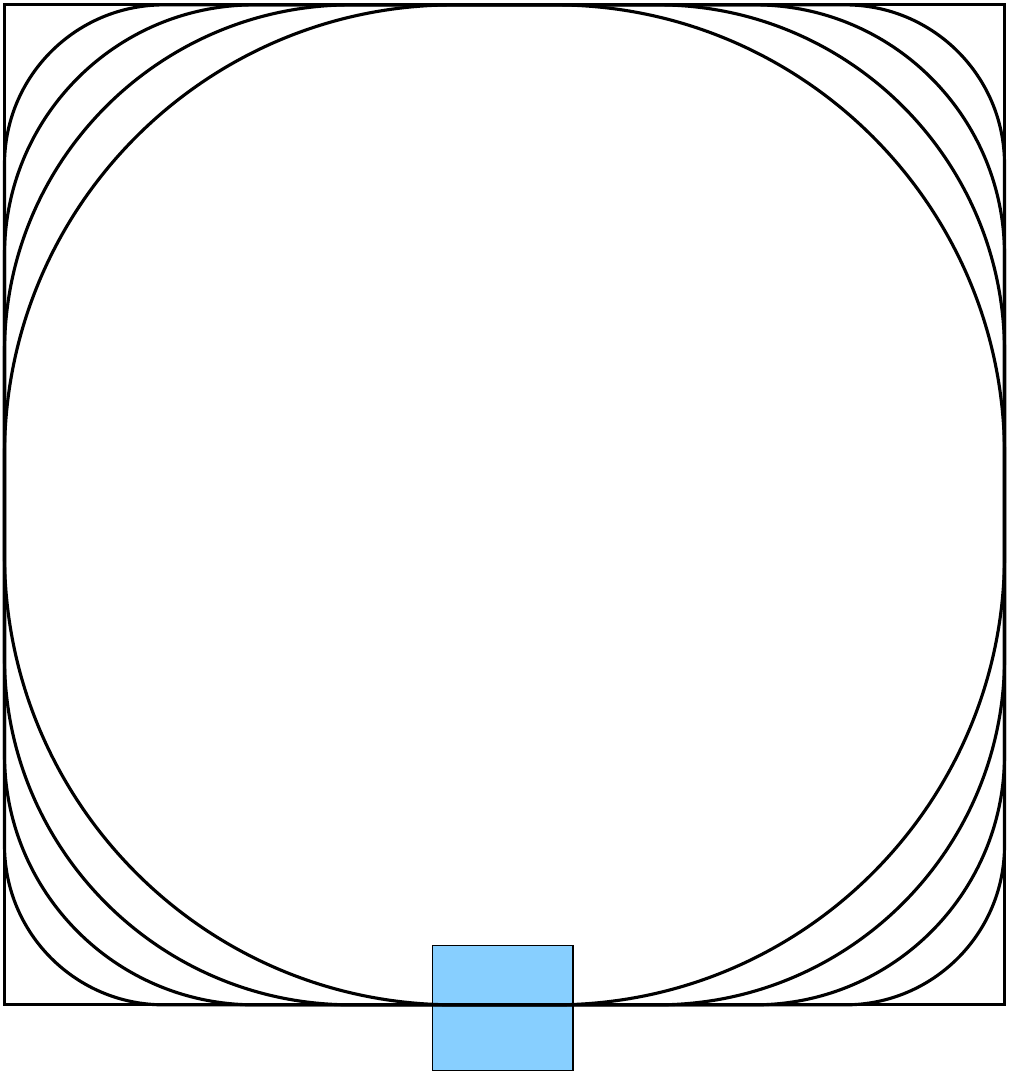}
\hskip1cm
\includegraphics[angle =0,height = 4.2cm]{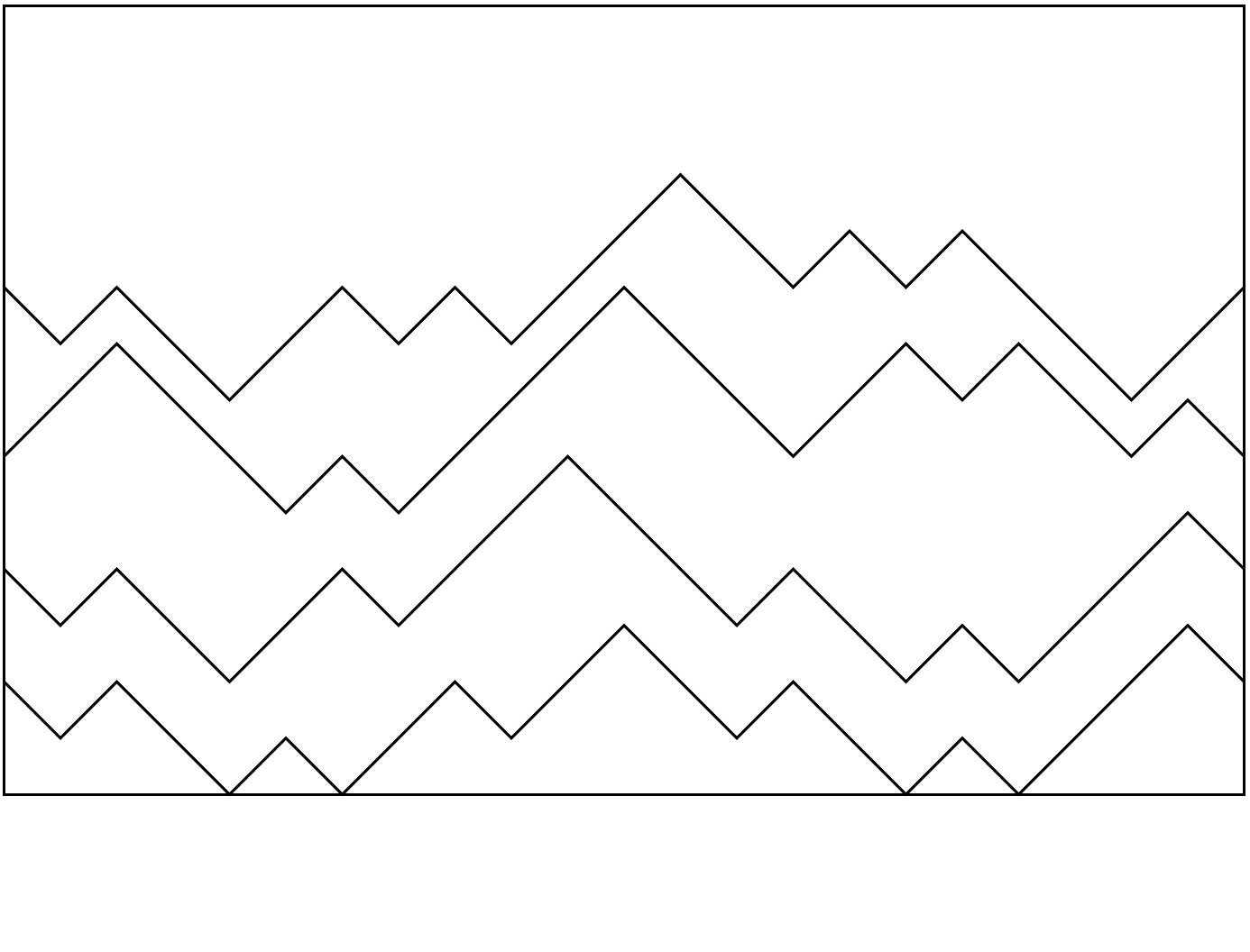}
\vspace{-0.1in}
\caption{Left: sketch of the top four nested limiting lines $\calL_h$, $h=1,\dots,4$ within the square $Q=[-1,1]^2$. The shaded region contains only macroscopically flat portions of the lines. Right: fluctuations of the lines on a scale of smaller order obtained by zooming in the shaded region.}
\label{fig:corners}
\vspace{-0.01cm}
\end{figure}

In  \cite{caputoetal2016} it was shown that at equilibrium,  the SOS surface above the wall is characterized by a uniquely defined ensemble  $\Gamma$, consisting of  the nested macroscopic contours (closed loops in the dual lattice, within $\Lambda_L$) $$\Gamma=\{\gamma_1\subseteq \dots\subseteq \gamma_n\subseteq\partial \Lambda_L\},$$ 
where $n=H(L)$, with the interpretation that $\gamma_h$ is the $(H(L)-h)$-th level line of the surface, so that $\varphi$ grows from at most $H(L)-h$ to at least $H(L)-h+1$ upon crossing $\gamma_h$. Moreover, it was shown that the contour ensemble satisfies a law of large numbers, that is if the box $\Lambda_L$ is rescaled to the square $Q=[-1,1]^2$, then when $L\to\infty$, the contours concentrate around a limiting shape consisting of infinitely many nested loops $\calL_1\subseteq\calL_2\subseteq\cdots\subseteq \partial Q$. 
The loops $\calL_h$ 
can be identified via constrained Wulff variational principles: each single $\calL_h$ 
is a rescaled 
Wulff plaquette such as 
that already studied in the context of $2D$ Ising model \cite{schonmannshlosman}.  
A related low temperature SOS-type model 
which, under appropriate rescaling,  features a stack of identical Wulff plaquettes was studied in \cite{IS2017}.

In our case  nested loops $\calL_h$ are strictly ordered by inclusion.   
However, apart from round pieces in the 
neighbourhood of the four corners of $Q$ 
(which are different for different plaquettes $\calL_h$-s), all $\calL_h$, $h\geq 1$, contain flat pieces, 
where they coincide with the boundary of $Q$; see Figure \ref{fig:corners}. In particular, there 
exists $u\in(0,1)$ (with $u=u(\beta)\to 1$ as $\beta\to\infty$) such that  
the portion $I_u=[-u,u]\times\{-1\}\subset\partial Q$ of the bottom side of $Q$ is contained in all loops $\calL_1,\calL_2,\dots$ 

To study fluctuations of the level lines of the surface it is then necessary to zoom in the shaded region from Figure \ref{fig:corners} and understand at what rate the contour ensemble $\Gamma$ converges to the flat limit there as $L\to\infty$. In this direction, it was shown in   \cite{caputoetal2016}  that the maximal distance of the top contour $\gamma_1$ from the bottom boundary of $Q$ is typically of order $L^{\frac13 + o(1)}$ and that fluctuations of that order appear on every subinterval of length $L^{\frac23 + o(1)}$, where $o(1)$ denotes a quantity vanishing as $L\to\infty$. 

In a first attempt,  we may approximate the $n$ paths in the above mentioned region by $n$ ordered height functions  $\Psi=\{\psi_1,\dots,\psi_n\}$ with free endpoints, where $\psi_i:[-L,L]\mapsto\bbN$, $\psi_i(x)\geq \psi_{i+1}(x)\geq 0$ for all $x$; see Figure \ref{fig:corners}. As discussed in \cite{caputoetal2016}, contour analysis based on cluster expansion techniques shows that the statistical weight of a configuration $\Psi$ of such lines is essentially given by 
\begin{equation}\label{eq:walks}
\exp\left(-\sum_{i=1}^n\left[\beta E(\psi_i) + \frac{ a \lambda^i}L \,A(\psi_i)\right] \right).
\end{equation}
Here $E(\psi_i)=\sum_{x}|\psi_i(x+1)-\psi_i(x)|$ denotes the energy cost of the $i$-th path, $a=a(\beta)>0$ and $\lambda=\lambda(\beta)>1$ are suitable constants, $A(\psi_i)=\sum_{x}\psi_i(x)$ is the area between the path $\psi_i$ and the bottom layer at height zero, while the term $a\lambda^i$ quantifies the entropic repulsion felt by the $i$-th path, which in these new coordinates becomes an effective attraction to the bottom. In agreement with their mutual order,  
the attraction felt by the $i$-th path  is stronger than the attraction felt by the $(i-1)$-th path.  
It should be remarked that in the  description \eqref{eq:walks} we are completely neglecting some nontrivial interaction terms between the paths $\psi_i$, which account for possible weak attractive and repulsive potentials along the polymer boundaries (pinning effects); while 
these terms should be 
indeed irrelevant if $\beta$ is sufficiently large, showing that this is actually the case can be a challenging problem; see \cite{caputoetal2016,ioffeshlosmantoninelli}.

\subsection{Ferrari-Spohn and Dyson-Ferrari-Spohn diffusions.} 
The expression \eqref{eq:walks} describes a growing number of ordered random walks above a wall with geometrically growing area tilts. 
The case of a single walk $n=1$ is an effective random walk model for the  critical pre-wetting problem in the $2D$ Ising model; 
see \cite{schonmannshlosman,velenik2004}. This case was recently studied in \cite{ioffeshlosmanvelenik}, where it was shown in 
particular that if $\ell=L^{\frac23}$, then the rescaled path 
\be{eq:FS-scaling}
x(t) = \frac1{\sqrt{\ell}}\,\psi_1(t\ell)\,,\qquad t\in\bbR,
\ee
as $L\to\infty$ converges weakly to the stationary Ferrari-Spohn diffusion, namely the reversible diffusion process 
on $\bbR_+$ with potential given by the logarithm of the Airy function, which was first introduced in \cite{ferrarispohn2005}. 
On the other hand, for fixed $n$, and for $\lambda=1$, the ensemble of random lines described by \eqref{eq:walks} has been analysed in \cite{ioffevelenikwachtel}, where it is shown that the vector of rescaled trajectories
 $$\underline x(t) = \frac1{\sqrt{\ell}}\,(\psi_1(t\ell),\dots,\psi_n(t\ell))\,,\qquad t\in\bbR,$$
 converges weakly to the stationary Dyson-Ferrari-Spohn diffusion, 
 that is the determinantal process on $\bbR_+^n$ corresponding to $n$ non-intersecting Ferrari-Spohn diffusions. 
 
 \subsection{The model and the result.}
 The case of $n$ ordered random walks with $n$ growing with $L$ (e.g.\  $n\propto\log L$ as in the 
 original setting 
 of SOS level lines) and $\lambda>1$
 %
 will be the subject of a separate paper. At this stage it is even unclear whether, under the Ferari-Spohn scaling 
 \eqref{eq:FS-scaling}, 
 such ensemble 
 has a meaningful limit. This is precisely the issue which we explore here. For simplicity, 
 and in order to stress main quantitative features of the phenomenon under consideration, 
 we shall consider a continuous analogue of the above model, where discrete random walk 
 paths are replaced by Brownian motion paths. The corresponding model then becomes the ensemble of 
 non-intersecting Brownian motions $X_i\geq X_{i+1}\geq 0$, $i=1,\dots,n$ on $[-T,T]$, for some $T>0$, with free endpoints, 
 deformed by the statistical weight
\begin{equation}\label{eq:brownians}
\exp\left(-\sum_{i=1}^n  a \lambda^i A(X_i) \right)
\end{equation}
where $a>0,\lambda>1$ and $A(X_i)=\int_{-T}^TX_i(s) ds$.

 If $n=1$ then  
 taking $T\to\infty$ one recovers 
(without further rescaling) the stationary Ferrari-Spohn diffusion. 
Similarly, 
if $n$ is fixed and $\lambda=1$, then $T\to\infty$ yields the stationary Dyson-Ferrari-Spohn diffusion. 
The invariant measure of the latter  is given by Slater determinants associated to the eigenfunctions of Airy differential operator
with Dirichlet boundary condition at the origin. 
 As $n\to\infty$
there are various (edge, bulk, $\dots$) scaling regimes, which were described
on the level of the corresponding determinantal point processes in \cite{Bornemann}.

In our case; $n$ growing with $T$ and $\lambda>1$,  
 no scaling is needed and 
the random line ensemble 
described by \eqref{eq:brownians} will have a 
very 
different limiting behavior, and a very different weak limit 
which would, with appropriate modifications related to the nature of area tilts,  
fall into the framework of $\bbN\times\bbR$-indexed  Brownian-Gibbs ensembles as developed
 in \cite{corwinhammond}.
 Indeed, it seems natural to conjecture that as $n\to\infty$ and $T\to\infty$ (regardless of the order), 
the process converges
to a unique $\bbN\times\bbR$-indexed random line ensemble, in particular
that  for every $k\in\bbN$ the top $k$ lines $X_1,\dots,X_k$ weakly
 converge to a stationary $k$-dimensional 
 process. 
For the moment,  
guessing the precise structure   of the limiting process remains an intriguing question. 
On the other hand, 
the route to proving   convergence per se is, in light 
of the stability results we derive here, rather clear - see Remark~\ref{rem:Tightness}  following \eqref{eq:max-exp-bound} below.    This issue will be addressed in a forthcoming separate 
paper.  

 Here   we focus on deriving  
 uniform 
 stability properties of the system of paths associated to \eqref{eq:brownians} as $n,T\to\infty$. 
 Namely, we prove  the following 
 strong confinement statement about the top path $X_1$  
(and hence about the whole stack which is sandwiched 
 between $X_1$ and the wall):
 Fix $\varepsilon > 0$ and let 
  $\left[\cdot\right]_+$ denote the positive part.
 Then 
 expectations of curved maxima
 \be{eq:max-exp-bound}
\sup_{T, n}\, \bbE \lb  \max_{t\in[-T,T]}\big[
 {
 	 X_1(t) -
 |t|^\varepsilon 
}
 ]_+ \rb < \infty, 
 \ee
 are uniformly bounded in $T$ and $n$. See Theorem~\ref{thm:max-control} for the precise statement. 

\begin{remark} 
	\label{rem:Tightness} 
	The bound \eqref{eq:max-exp-bound} 
	 paves the way for importing techniques and ideas developed in \cite{corwinhammond}, in particular for deriving 
	 appropriate adjustments of Proposition~3.5 and of the tightness arguments employed for the proof of
	 Proposition~3.6 
	 of the latter work. Indeed, fix $k\in \bbN$ and an interval $(-a, a)$ 
 and, for $n>k$ and $T >a $,  consider top $k$ paths $X_1,\dots,X_k$ of  the line ensemble of $n$ 
 non-intersecting Brownian motions on $[-T, T]$ under 
geometric area tilts \eqref{eq:brownians}.   
But then \eqref{eq:max-exp-bound}	
	means two things: First of all, by Brownian scaling and stochastic domination - 
	see Subsections~\ref{sub:BR} and \ref{sub:SD}  below, 
	  the height 
	of the $(k+1)$-st path $X_{k+1}$ is, uniformly 
	in $n$ and $T$, under control in the sense that it sits below a random integrable shift of 
	the appropriate rescaling of $ t^{\epsilon}$. Next, given a 
	realization of $X_{k+1}$, the top paths $X_1,\dots,X_k$ could be, by Brownian-Gibbs property,  
	resampled  according to Brownian bridge measures modified by exponential weights \eqref{eq:brownians}. 
	But these weights are, by the very same \eqref{eq:max-exp-bound}, also subject to a uniform control. 
	Resampling with respect to reference Brownian bridge measures (that is with zero area tilts)  
	is precisely the procedure employed for
	Airy line ensembles in \cite{corwinhammond}.
	As a result many probabilistic estimates 
	could be, up to uniformly bounded corrections,  inherited from the resampling estimates developed
	in \cite{corwinhammond}. 
	\end{remark}

\bigskip
We proceed with precise notation for the polymer measures
we study here. 
\subsubsection{\bf Notation for underlying Brownian motion and Brownian bridges.}  
In the sequel we shall use the same notation for path measures of underlying Brownian motion and 
Brownian bridges and for expectations with respect to these
path measures. 
For  $S < T$ and $x\in\bbR$, 
let ${\mathbf P}^x_{S,T} $ be the  path measure of the Brownian 
motion $X$ on $[S,T]$ which starts at $x$ at time $S$; $X (S ) = x$. 
We can record ${\mathbf P}^x_{S,T} $ as follows: 
\be{eq:BM-BB} 
{\mathbf P}^x_{S,T} \lb F (X ) \rb = \int {\mathbf B}^{x, y}_{S,T} \lb F ( X )\rb \dd y
\ee
where ${\mathbf B}^{x,y}_{S,T} $ the  {\em unnormalized} path measure of the Brownian 
bridge  $X$ on $[S,T]$ which starts at $x$ at time $S$ and ends at 
$y$ at time $T$; $X (S ) = x,\, X (T )= y$.  
{
In this notation 
$${\mathbf B}^{x,y}_{S,T}  (1 ) = \tfrac1{\sqrt{2\pi(T-S)}}\,{\rm e}^{- \frac{(y-x)^2}{2(T-S )}}.$$ 
}
\smallskip 

For an 
$n$-tuple $\ux\in\bbR^n$, set 
\[
{\mathbf P}^{\ux}_{S,T} =
{\mathbf P}^{x_1}_{S,T} \otimes {\mathbf P}^{x_2}_{S,T} \otimes \cdots \otimes 
{\mathbf P}^{x_n}_{S,T} .
\]
Similarly for 
$n$-tuples $\ux, \uy \in\bbR^n$, set 
\[
{\mathbf B}^{\ux , \uy}_{S,T} =
{\mathbf B}^{x_1, y_1 }_{S,T} \otimes {\mathbf B}^{x_2 , y_2}_{S,T} 
\otimes \cdots \otimes 
{\mathbf B}^{x_n , y_n }_{S,T} .
\]
For symmetric intervals $S=-T$ we shall employ a reduced 
notation  ${\mathbf P}^{\ux}_{T}$, ${\mathbf B}^{\ux , \uy}_{T}$ and 
so on. 

\subsubsection{\bf Partition functions and polymer measures  with geometric area tilts.}
Given a function $h$, the signed $h$-area under the trajectory of $X$ is defined as
\be{eq:T-area} 
\calA_{S,T}^h \lb X\rb = \int_{S}^T h (t ) X (t)\dd t .
\ee
We shall drop the superscript if $h\equiv 1$ and use $\calA_{S, T} \lb X\rb$ accordingly. 
For  $n\in\bbN$ define
\be{eq:Aplus}
\bbA_n^+ = \{ \ux \in \bbR^n \,:\, x_1 >\dots > x_n > 0\} .
\ee
Polymer measures which we consider in the sequel are always  
concentrated on the set $\Omega^{+}_{n, S, T}$ of  $n$-tuples  $\uX$, 
\be{eq:Omega-Set} 
\Omega^{+}_{n, S, T} = \lbr \uX~:~ \uX (t )\in \bbA_n^+\ \ \forall\, t\in [S,T]\rbr
\ee
Following our convention we shall write $\Omega^{+}_{n, T} =\Omega^{+}_{n, -T, T}$.

Given 
 $n\geq 1$,  $a>0$ and $\lambda >1$ consider now the
partition 
functions
\be{eq:PF-BC-T} 
\begin{split} 
 Z_{n, T}^{ \ux , \uy } (a , \lambda ) 
&\df  {\mathbf B}^{\ux , \uy}_{T} 
\lb 
\1_{\Omega_{n, T}^+}
{\rm e}^{-\sum_1^n a\lambda^{i-1}\calA_{T} (X_i )}
\rb \\ 
&{\rm and}\\ 
{\mathcal Z}_{n ,  T} (a, \lambda )
&\df 
\int_{\bbA_{n}^+} 
\int_{\bbA_{n}^+}
Z_{n, T}^{ \ux , \uy}  (a, \lambda)  
\dd\ux 
\dd 
\uy .
\end{split} 
\ee
Partition functions ${\mathcal Z}_{n,T} (a, \lambda )$  
give rise to polymer measures 
$\bbP_{n, T}\left[ \cdot  ~|a, \lambda  \right]$. 
with geometric area tilts $\rho_i \equiv a\lambda^{i-1}$.   
 Namely, 
\be{eq:PolMeas} 
\bbP_{n, T}
\left[ F (\uX ) ~|a, \lambda  \right]  
\ 
\df \frac{1}{{\mathcal Z}_{n ,  T} (a, \lambda )} 
\int_{\bbA_{n}^+} 
\int_{\bbA_{n}^+}
{\mathbf B}^{\ux , \uy}_{T} 
\lb F (\uX )
\1_{\Omega_{n, T }^+}
{\rm e}^{-\sum_1^n \calA_{T}^{\rho_i} (X_i )} 
\rb 
\dd \ux \dd\uy 
\ee

\subsubsection{\bf Main result.} The main result of this  note could be formulated as follows:  
\begin{theorem} 
\label{thm:main} 
Let $\calP_{n, T}\left[ \cdot  ~|a, \lambda  \right]$ 
be the (one-dimensional) distribution of the position of the top path $X_1 (0 )$ under 
$\bbP_{n, T}\left[ \cdot  ~|a, \lambda  \right]$. That is  
\[ 
 \calP_{n, T}\left[  f({X}) ~|a, \lambda  \right]
 = \bbP_{n, T} \left[ f (X_1 (0))  ~|a, \lambda  \right],
\]
for bounded measurable $f:\bbR\mapsto\bbR$. 
Then for any fixed $a>0$ and  $\lambda >1$
 the family of one-dimensional 
 distributions $\lbr \calP_{n, T}\left[ \cdot  ~|a, \lambda  \right]\rbr_{n, T}$ 
 is tight. In other words the top path does not fly away as the number of
 polymers and the length of their horizontal span grow. 
\end{theorem}
Theorem~\ref{thm:main} is an immediate consequence of a much stronger confinement
statement for the {\em curved} maximum of the whole path - see Theorem~\ref{thm:max-control} below.
\subsection{Structure of the proof.} The proof is built upon  a recursion which relies on 
the Brownian scaling 
and on stochastic domination for (a more general class of) 
polymers with area tilts:
\smallskip

\subsubsection{\bf Brownian scaling.} 
\label{sub:BR}
Consider the following mapping of an $n$-tuple 
$\uX$ of paths  on an interval $[-\lambda^{2/3}T , \lambda^{2/3}T]$ 
to $n$-tuple $\uY$ of paths on  $[-T , T]$: 
\be{eq:BrownScale}  
\uY (\cdot  ) = \frac{1}{\lambda^{1/3}} \uX (\lambda^{2/3}\cdot ) .
\ee
The next lemma states that if $\uY$ is related to $\uX$ via \eqref{eq:BrownScale}, then 
$\uY$ has distribution $\bbP_{n, T} \left[ \cdot   ~|a\lambda , \lambda  \right]$ if and only if 
$\uX$ has distribution $\bbP_{n, T\lambda^{2/3}} \left[ \cdot   ~|a , \lambda  \right]$. 
\begin{lemma} 
\label{lem:scaling} 
For all $n,T$, $a,\lambda$, and $ \ux , \uy$
\be{eq:BScale-1} 
Z_{n,  T\lambda^{2/3}}^{ \ux , \uy}( a, \lambda) = \lambda^{-\frac{n}3}
Z_{n, T}^{\lambda^{-1/3} \ux , \lambda^{-1/3}\uy} ( a\lambda , \lambda). 
\ee
Furthermore, for any bounded measurable function $F$ on $\Omega_{n,T}^+$,
\be{eq:BScale-11} 
\bbP_{n, T}
\left[ F (\uX ) ~|a\lambda, \lambda  \right]  = \bbP_{n, T\lambda^{2/3}}
\left[ F (\uX^{(\lambda^{2/3})} ) ~|a, \lambda  \right],
\ee 
where $\uX^{(\lambda^{2/3})}=\frac{1}{\lambda^{1/3}} \uX (\lambda^{2/3}\cdot ) $. 
\end{lemma}
\begin{proof}
For any $\gamma>0$,  
consider the map $\Omega_{n,\gamma T}^+\mapsto\Omega_{n,T}^+$ defined by
\[
\uX^{(\gamma)}(t) = \gamma^{-\frac12}\uX(\gamma t)\,,\qquad t\in[-T,T].
\] 
The normalized Brownian bridge ensemble  
\[
\Gamma^{\ux,\uy}_T = \frac{{\mathbf B}^{\ux , \uy}_{T} }{{\mathbf B}^{\ux , \uy}_{T} (1)},
\]
has the Brownian scaling property:
\be{eq:Bscale-0}
\Gamma^{\ux,\uy}_{\gamma T}[G(\uX^{(\gamma)})] = \Gamma^{\gamma^{-\frac12}\ux,\gamma^{-\frac12}\uy}_{ T}[G(\uX)]\,,\qquad \gamma>0,
\ee
where $G$ is any bounded measurable function on $\Omega_{n,T}^+$.
We apply \eqref{eq:Bscale-0} with $$G(\uX)= \1_{\Omega_{n, T}^+}(\uX)
\,{\rm e}^{-\sum_1^n a\lambda^{i-1}\calA_{T} (X_i )},$$ and $\gamma=\lambda^{2/3}$. Since, 
{the area under the path scales as}
$$
\frac1{\lambda}\int_{-T\lambda^{2/3}}^{T\lambda^{2/3}}X_i(t)dt=\int_{-T}^T\frac{X_i(t\lambda^{2/3})}{\lambda^{1/3}}dt,
$$
this yields \eqref{eq:BScale-1}. The proof of \eqref{eq:BScale-11} is similar, with $$ G(\uX)= F(\uX)\1_{\Omega_{n, T}^+}(\uX)
\,{\rm e}^{-\sum_1^n a\lambda^{i-1}\calA_{T} (X_i )}.$$
\end{proof}
\smallskip 

\subsubsection{\bf A general class of polymers with area tilts.}
\label{sub:pmeas}
Let us say that two continuous\footnote{Here and below the assumption of continuity is for convenience.}
functions $f$ and $g$ on $[-T, T]$ satisfy $f\prec g$ if $ f (t )\leq g (t )$ 
for any $t\in [-T , T]$. By construction, 
if $\uX\in \Omega_{n, T }^+$, then $0\prec X_n\prec  X_{n-1} \prec \dots \prec X_1$. For every $n\in \bbN$ 
and $T>0$, let us consider the following 
general class $\bbP_{n, T}^{\ux , \uy}\left[\, \cdot | h_- , h_+, \urho\right]$ of
polymer measures which is parametrized by: 
\begin{description} 
 \item[a] Boundary conditions 
 $\ux, \uy \in \bbA_n^+$. 
 \item[b] Two non-negative continuous functions $h_-\prec h_+$ on $[ -T , T]$, which are called 
 the floor $h_-$ and the ceiling $h_+$. 
 \item[c] A tuple of $n$ (not necessarily ordered) positive 
 continuous functions  $\urho = \lbr \rho_1, \dots , \rho_n\rbr$
 which are called area tilts. 
\end{description}
Then, setting 
 \be{eq:Om-h} 
\Omega_{n, T }^{\uh} = \Omega_{n, T }^+\cap \lbr  h_- \prec X_n \prec X_1 \prec h_+\rbr, 
\ee 
define:
 \be{eq:mu-meas} 
 \bbP_{n, T}^{\ux , \uy}\left[\dd\uX  | h_- , h_+, \urho\right] \,
 \propto \, 
 {\rm e}^{- \sum_1^n \calA_T^{\rho_i} (X_i )} \1_{\Omega_{n, T }^{\uh}}
 {\mathbf B}_T^{\ux , \uy }\lb \dd \uX \rb .
 \ee
The corresponding partition function is denoted 
$Z_{n, T}^{\ux , \uy}( h_- , h_+, \urho )$. 

There is a  version of \eqref{eq:mu-meas} which permits  more general 
boundary conditions: Let $\unu$ and $\ueta$ be $n$-tuples of 
 functions on $\bbR_+$. 
For $\ux\in \bbA_n^+$ set $\unu (\ux ) =\sum_1^n \nu_i (x_i )$, and $\ueta (\ux ) =\sum_1^n \eta_i (x_i )$. 
Similarly, set $\calA_T^{\urho} (\uX ) = \sum_1^n \calA_T^{\rho_i} (X_i )$. 
Then, 
\be{eq:mu-meas-bc} 
 \bbP_{n, T}^{\unu , \ueta}\left[\dd\uX  | h_- , h_+, \urho\right] \,
 \propto \, 
 \int_{\bbA_n^+} \int_{\bbA_n^+} {\rm e}^{- \calA_T^{\urho} (\uX )} \1_{\Omega_{n, T }^{\uh}}\lb \uX\rb 
 {\rm e}^{-\unu (\ux )}{\mathbf B}_T^{\ux , \uy }\lb \dd \uX \rb{\rm e}^{-\ueta (\uy )}\dd\ux \dd \uy .
 \ee
The corresponding partition function is denoted 
$\calZ_{n, T}^{\ueta , \unu}( h_- , h_+, \urho )$. 
{This notation could be made formally compatible with \eqref{eq:PF-BC-T} as follows:  if we define 
	$\delta_x (y ) = \infty\1_{y\neq x}$, and,  for a 
	tuple $\ux$ define the $n$-tuple 
	$\udelta_{\ux} = \lbr \delta_{x_1}, \dots ,\delta_{x_n}\rbr$, then 
	$Z_{n, T}^{\ux , \uy}( h_- , h_+, \urho ) = 
	\calZ_{n, T}^{\udelta_{\ux} , \udelta_{\uy}}( h_- , h_+, \urho ). 
	$}

\bigskip

\noindent
{\bf Reduced notation.}
In the sequel we shall, unless this creates a confusion, employ the following reduced
notation:
If $\ueta,\unu$ are identically zero, 
we shall drop them  from the notation. We refer to this as the case of free (or empty) boundary conditions.  Similarly we shall drop from the notation the floor $h_-$ 
whenever $h_- \equiv 0$ and the ceiling $h_+$ whenever $h_+\equiv \infty$. 
Furthermore, we shall write $a, \lambda$ instead of $\urho$ whenever $\urho = \lbr a, a\lambda , \dots ,a\lambda^{n-1}\rbr$.
Finally we shall drop $n$
whenever talking about just one polymer, $n=1$. In this way we make new notation compatible 
with \eqref{eq:PolMeas}. 

Below we shall tacitly assume that boundary conditions $\unu , \ueta$ are chosen in such a 
way that the corresponding polymer measures are well defined. 
In Appendix~\ref{sub:Exists}  this assumption will be justified for a class of boundary conditions 
$\unu, \ueta$, including free
boundary conditions.

\subsubsection{\bf Stochastic domination.} 
\label{sub:SD} 
Equip $\Omega_{n, T }^{+}$ with the partial  order $\prec $, defined by $$\uX\prec \uY\;\;\text{iff}\;\; X_i\prec Y_i\,,\;\text{for all }\,i,$$ and let $\fkg $ denote the associated notion of stochastic domination of probability measures. 
\begin{lemma}\label{lem:SD}
For any $n$ and $T$, $h_-\prec g_-$, $h_+ \prec g_+$, $\urho\succ \ukappa$, the following holds. 
If,  $\ux\prec \uu$ and $\uy\prec \uv$, then 
\be{eq:FKG-2} 
\bbP_{n, T}^{\ux, \uy}\left[\cdot   | h_- , h_+, \urho\right]\fkg 
\bbP_{n, T}^{\uu , \uv}\left[\cdot  | g_- , g_+, \ukappa \right].
\ee 
Moreover, for an $n$-tuple $\uchi = \lbr \chi_1 ,\dots ,\chi_n\rbr$ of smooth boundary condition let $\uchi^\prime$ be the $n$-tuple of corresponding first derivatives. Then,
\be{eq:FKG-1} 
\bbP_{n, T}^{\uxi , \uzeta}\left[\cdot   | h_- , h_+, \urho\right]\fkg 
\bbP_{n, T}^{\unu , \ueta}\left[\cdot  | g_- , g_+, \ukappa \right], 
\ee 
whenever, 
 $h_-\prec g_-$, $h_+ \prec g^+$, $\urho\succ \ukappa$ and, 
  both $\underline\xi'
  \succ \unu'$ and $\underline\zeta'
  \succ \ueta'$. 
   In particular, \eqref{eq:FKG-1} holds if 
$\uxi = \unu$ and $\uzeta = \ueta$ (by approximation without any assumptions on smoothness).
\end{lemma}
The proof of Lemma \ref{lem:SD} is given in Section \ref{sub:SDproof} below.

\subsubsection{\bf A recursion.}\label{subsec:recursion} 
Consider our  polymer measure $\bbP_{n+1, T}
\left[ \cdot ~|a, \lambda  \right]$. 
Let us record $\uX = (X_1 , \tilde \uX)$, where $\tilde \uX = ( X_2 , \dots , X_{n+1})$. 
It is easy to see that the conditional, on $X_1$, distribution of $\tilde\uX$ is precisely 
$\bbP_{n, T}\left[\cdot   | X_1 ,   a\lambda , \lambda\right]$ with ceiling
$X_1$. 
By \eqref{eq:FKG-1} it is stochastically dominated by 
$\bbP_{n, T}\left[\cdot   |   a\lambda ,\lambda\right]$. 
Therefore,  if 
$F$ is  a non-decreasing function on $\Omega_{n,T}^+$, 
\be{eq:FKG-n-bot} 
\bbE_{n+1, T}\left[ F (\tilde{\uX} )  |  a, \lambda \right]
 \leq 
 \bbE_{n, T}\left[ F ({\uX} )  |  a\lambda, \lambda\right]. 
\ee
Let $f$ be a non-decreasing functional of $X_1$. Then, conditioning on $X_2$,
\be{eq:dom1} 
\bbE_{n+1, T}\left[ f (X_1 )  \, \big|\, X_2 \,|\, a, \lambda \right] = \bbE_T\left[   f (X) \big|\, X_2 , a\right] 
\ee
is, by \eqref{eq:FKG-1},  a non-decreasing functional of the floor $X_2$. 
Applying  \eqref{eq:FKG-n-bot} with $F(\tilde \uX)=\bbE_T\left[   f (X ) \big|\, X_2 , a\right] $ and 
then using the scaling 
property  
\eqref{eq:BScale-11}, 
this means: 
\be{eq:FKG-reduction} 
\begin{split}
\bbE_{n+1,T} \left[  f (X_1 ) |a, \lambda \right] &\leq \bbE_{n,T}\left[ 
\bbE_T \left[  f (X ) \big|\, X_1, a \right]\big| a\lambda , \lambda \right] \\
&\stackrel{\eqref{eq:BScale-11} }{=} 
\bbE_{n,T\lambda^{2/3}}\left[ 
\bbE_T \left[  f (X ) \big|\, \frac{X_1 (\lambda^{2/3}\cdot )}{\lambda^{1/3}}, a \right]\big| a , \lambda \right]. 
\end{split}
\ee
Inequality \eqref{eq:FKG-reduction} sets up the stage for various recursions which eventually 
lead to tightness statements like the one formulated in Theorem~\ref{thm:main}. In particular, 
a natural choice of $f (X) = \max_{t\in [-T, T]} X (t )$ in \eqref{eq:FKG-reduction} leads to  a 
proof of tightness (in $n$) of the maxima  for each $T$ fixed, but clearly is not suitable for proving tightness
uniformly in $T$. In order to prove Theorem~\ref{thm:main} we introduce a kind of {\em curved} maximum, 
and use \eqref{eq:FKG-reduction} to control it uniformly in $n$ and $T$. Both the usual and the curved maxima
are discussed in the  subsequent sections. 

{
\begin{remark}
	\label{rem:domination}
Let us also remark that the above reasoning can be used to control the height of the $k$-th path $X_k$ in terms of the top path $X_1$. Indeed, as in \eqref{eq:FKG-n-bot} one has that $X_{k+1}$ under $\bbP_{n, T}\left[ \cdot |  a, \lambda \right]$ is stochastically dominated by $X_{k}$ under $\bbP_{n-1, T}\left[ \cdot |  a\lambda, \lambda \right]$.
Iterating, it follows that  $X_{k+1}$ under $\bbP_{n, T}\left[ \cdot |  a, \lambda \right]$ is stochastically dominated by $X_{1}$ under $\bbP_{n-k, T}\left[ \cdot |  a\lambda^k, \lambda \right]$. By \eqref{eq:BScale-11} therefore $X_{k+1}$ under $\bbP_{n, T}\left[ \cdot |  a, \lambda \right]$ is stochastically dominated by $\lambda^{-k/3}X_1$ where $X_1$ has law $\bbP_{n-k, T\lambda^{2k/3}}\left[ \cdot |  a, \lambda \right]$.
\end{remark}
}

In implementing the recursions, we shall repeatedly use the following easily verified fact, which crucially depends on the linear structure 
of area tilts: For any number $\xi > 0$, any tilt $\rho$ and any floor $h$, 
the distribution of $(Y- \xi)$, when $Y$ is distributed according to $\bbP_T\left[\cdot |\xi+ h, \rho\right]$, is the same as the distribution 
of $Y$ under $\bbP_T\left[\cdot | h, \rho\right]$.

\section{Tightness of maxima.} \label{sec:maxima}
In this section we shall prove the following proposition, 
which can be considered as a warm-up towards the much stronger 
statement of Theorem~\blue{\ref{thm:max-control}} below: 
\begin{proposition} 
 \label{prop:main} 
For any $a>0$, $\lambda >1$ and $T$ fixed, there exists a constant $C(a,\lambda,T)$ such that
\be{eq:prop-main} 
\lim_{n\to\infty} \,\bbE_{n, T}\left[ \max_{t\in[-T , T]} X_1 (t )\big|\, a, \lambda\right]= C(a,\lambda,T)<\infty .
\ee
\end{proposition}
\begin{proof}
Define
\be{eq:maxnT} 
M_{n,T}(a,\lambda) \df \bbE_{n, T}\left[ \max_{t\in[-T , T]} X_1 (t )\big|\, a, \lambda\right]. 
\ee
When $n=1$, we simply write $M_T(a)$ for $M_{1,T}(a,\lambda)$. Clearly, $ M_T(a)<\infty$ for all $a,T>0$. 
As in the first line of \eqref{eq:FKG-reduction} we obtain
\be{eq:maxnT2} 
M_{n+1,T}(a,\lambda) \leq \bbE_{n, T}\left[ \bbE_{T}\left[\max_{t\in[-T , T]} X (t )\big|\,X_1,a\right]\big|\, a\lambda, \lambda\right]. 
\ee
Using stochastic domination and the remark at the end of Section \ref{subsec:recursion}, we may replace the floor  $X_1$ by the constant floor $\xi=\max_{t\in[-T , T]} X_1(t)$ to obtain
\[
\bbE_{T}\left[\max_{t\in[-T , T]} X (t )\big|\,X_1,a\right]
\leq M_T(a)+\max_{t\in[-T , T]} X_1(t ). 
\]
Thus \eqref{eq:maxnT2}  implies the recursive estimate
\be{eq:maxnT3} 
M_{n+1,T}(a,\lambda) \leq  M_T(a)+ M_{n,T}(a\lambda,\lambda). 
\ee
Iterating, for any $n\in\bbN$:
\be{eq:maxnT3a} 
M_{n,T}(a,\lambda) \leq  \sum_{k=0}^{n-1}
M_T(a\lambda^k). 
\ee
Stochastic domination implies also that the sequence $M_{n,T}(a,\lambda)$ is monotone in $n$ and therefore
\be{eq:maxnT3b} 
\lim_{n\to\infty} M_{n,T}(a,\lambda)=\sup_{n\in\bbN} M_{n,T}(a,\lambda)\leq  \sum_{k=0}^\infty
M_T(a\lambda^k). 
\ee
From the scaling relation \eqref{eq:BScale-11}:
\be{eq:maxnT4} 
M_T(b) = \frac1{b^{1/3}} M_{Tb^{2/3}}(1)\,,\qquad b>0.
\ee
From \eqref{eq:maxnT4} it follows that the sum in \eqref{eq:maxnT3b} is finite if e.g.\ 
\be{eq:maxnT5} 
M_T(1) \leq C T^\alpha\,,
\ee
for some constants $C>0,\alpha\in(0,\tfrac12)$, for all $T\geq 1$. The bound \eqref{eq:maxnT5} 
can be derived  from the explicit representation \eqref{eq:Zt-expAiry} for the partition functions. 
Since we prove much stronger estimates in the next section we omit the details here. Notice in particular that 
the argument for the estimate \eqref{tail_4} below actually allows us to prove that 
\be{eq:maxnT6} 
M_T(1) \leq C \log (1+T),
\ee
for all $T\geq 1$. 
\end{proof}

\section{Curved maximum and Uniform tightness.} 
Let us start with explaining our notion of curved maxima. 
Let $\varphi (t ) = \abs{t}^\alpha$ with $\alpha \in (0, \frac{1}{2} )$. 
Given a continuous function $h$ on $[-T , T]$ define (see Figure~\ref{fig:a}) 
\be{eq:xi-func} 
\xi^T_{\varphi} ( h) = \min\lbr y\geq 0\,  :\, y+\varphi \succ h\rbr = \max_{t\in[-T,T]} \left[ h (t ) - \varphi (t )\right]_+. 
\ee
Informally, $\xi^T_{\varphi} ( h)$ is the minimal amount to lift $\varphi$ so that 
it will stay above $h$. We think of $\xi^T_{\varphi} ( h)$ in terms of the curved maximum of $h$ on
$[ -T, T]$. 

\begin{figure}[h]
\begin{overpic}[scale=0.6]{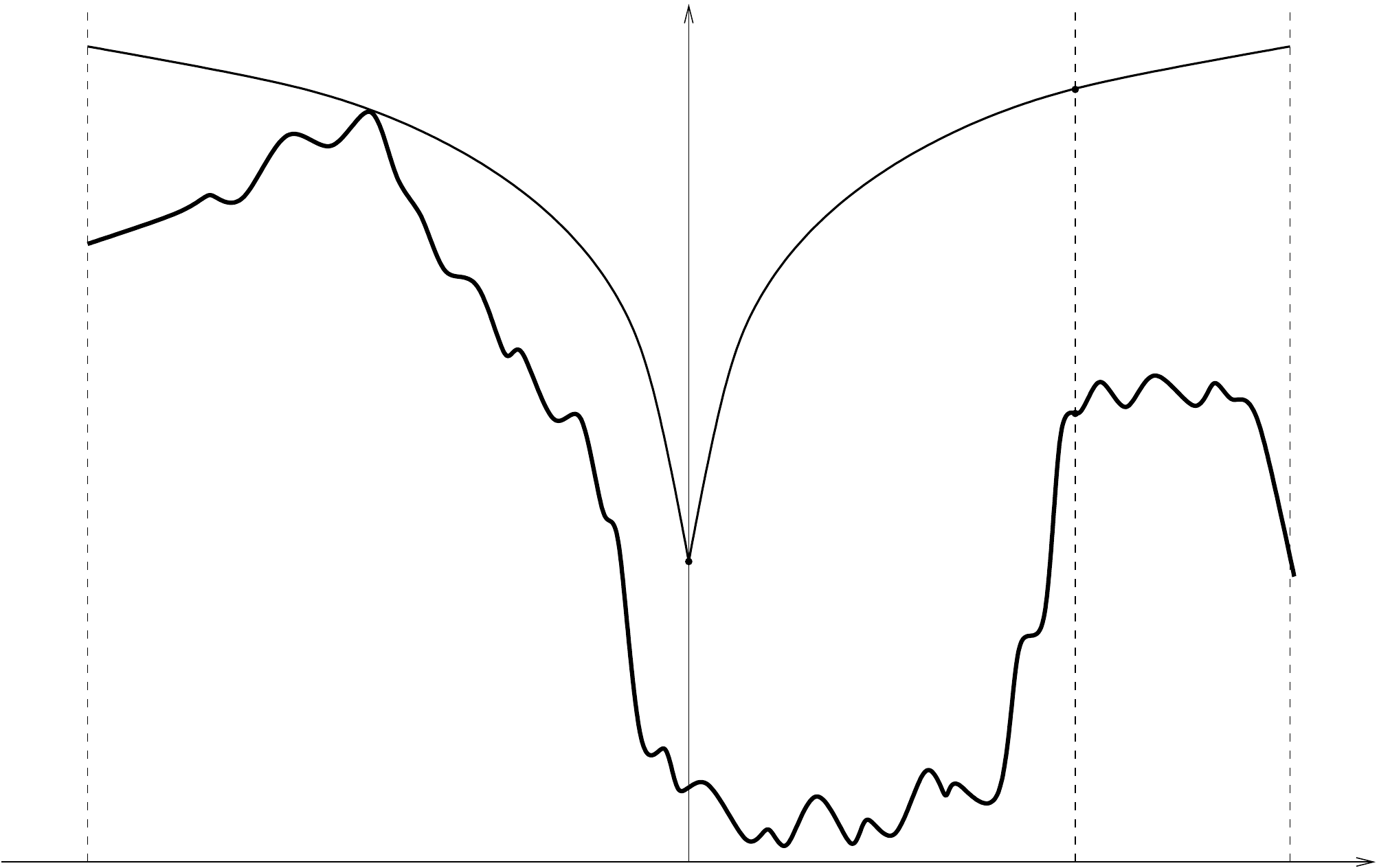}
\put(51,21){\scalebox{1}{$\xi^T_\varphi(h)$}}
\put(62,57.5){\scalebox{.9}{$\xi^T_\varphi(h)+\varphi(t)$}}
\put(72,34){\scalebox{1}{$h(t)$}}
\put(77.5,-3){\scalebox{1}{$t$}}
\put(92.5,-3){\scalebox{1}{$T$}}
\put(2.4,-3){\scalebox{1}{$-T$}}
\end{overpic}
\label{fig:a}
\caption{ The curved maximum $\xi_{\varphi}$.} 
\vspace{-0.25cm}
\end{figure}
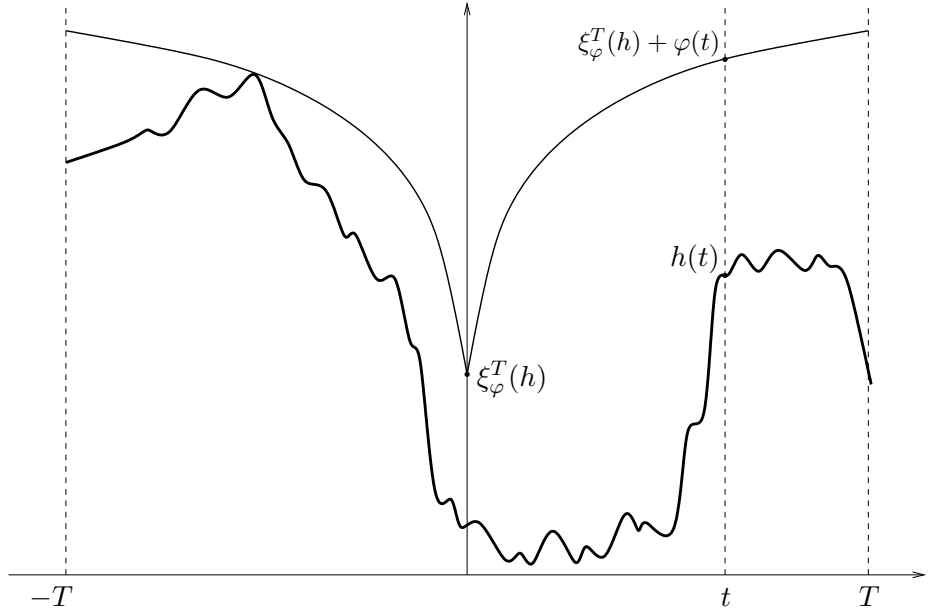

%

Theorem~\ref{thm:main} is an immediate consequence of the following result:
\begin{theorem} 
\label{thm:max-control}
For any $a>0,\lambda>1$ and $\alpha\in(0,\tfrac12)$:
\be{eq:max-control}
 \sup_{T}\max_n \bbE_{n, T} \left[  \xi^T_{\varphi} ( X_1 )| a, \lambda\right]  <\infty .
 \ee
\end{theorem} 
Our proof of Theorem~\ref{thm:max-control} comprises several steps. The first one is a reduction to the 
key fact  \eqref{eq:key-phi} below,  about single polymers above concave floors.

\subsection{Reduction to a statement about single polymers above concave floors} 
The functional $h\mapsto \xi^T_{\varphi} (h)$ is increasing, and we can take advantage of 
\eqref{eq:FKG-reduction}: 
\be{eq:OneStepBound} 
\bbE_{n+1,  T}\left[  \xi^T_{\varphi} ( X_1 )\, \big| a, \lambda \right]  \leq 
\bbE_{n, \lambda^{2/3} T }
\left[ 
\bbE_T \left[   \xi^T_{\varphi} (Y ) \big|\, \frac{X_1 ( \lambda^{2/3}\cdot )}{\lambda^{1/3}} , a \right]
\big| a, \lambda\right] .
\ee
Set
\be{eq:varphi}
 \varphi_\lambda (t ) = \frac{1}{\lambda^{1/3}} \varphi\lb \lambda^{2/3} t \rb = 
 \frac{1}{\lambda^{\frac{1}{3} (1- 2\alpha )}} 
 \varphi (t) =: \frac{1}{\lambda^\beta} \varphi (t ).
\ee 
By the definition of $\xi^T_{\varphi}$, 
\be{eq:X1-bound}
 \frac{1}{\lambda^{1/3}} X_1 ( \lambda^{2/3} t  ) \leq
 \frac{1}{\lambda^{1/3}}\lb  \xi^{T\lambda^{2/3}}_{\varphi} (X_1 ) + \varphi  (\lambda^{2/3} t)\rb = 
 \frac{1}{\lambda^{1/3}} \xi^{T\lambda^{2/3}}_{\varphi} (X_1 ) + \varphi_{\lambda } (t) , 
\ee
Hence, the stochastic domination \eqref{eq:FKG-1} implies 
\begin{align}\label{eq:FKG-sum-bound} 
&\bbE_T \left[   \xi^T_{\varphi} (Y ) \big|\, \frac{X_1 ( \lambda^{2/3}\cdot )}{\lambda^{1/3}} , a \right] 
\nonumber\\&\qquad \leq \bbE_T \left[  \xi^T_{\varphi} (Y  )  \big|\,\frac{\xi^{T\lambda^{2/3}}_{\varphi} (X_1 )}{\lambda^{1/3}} + \varphi_{\lambda }, a \right] = 
\frac{\xi^{T\lambda^{2/3}}_{\varphi} (X_1 )}{\lambda^{1/3}} + \bbE_T \left[   \xi^T_{\varphi} (Y ) \big|\, \varphi_{\lambda }, a \right]. 
\end{align}
In the last equality above we  
relied on the linearity of area tilts, see
the observation at the end of Section \ref{subsec:recursion}.
 Going back to \eqref{eq:OneStepBound} we conclude: 
\be{eq:OneStepBound-1}
\bbE_{n+1,  T}\left[  \xi^T_{\varphi} ( X_1 )\, \big| a, \lambda \right]
 \leq  
\frac{1}{\lambda^{1/3}} \bbE_{n, T \lambda^{2/3}}\left[ \xi^{T\lambda^{2/3}}_{\varphi} ( X_1 )\, \big| a, \lambda\right]  + 
\bbE_T \left[   \xi^T_{\varphi} (Y ) \big| \varphi_\lambda ,a \right] .
\ee
\begin{figure}[h]
\begin{overpic}[scale=0.6]{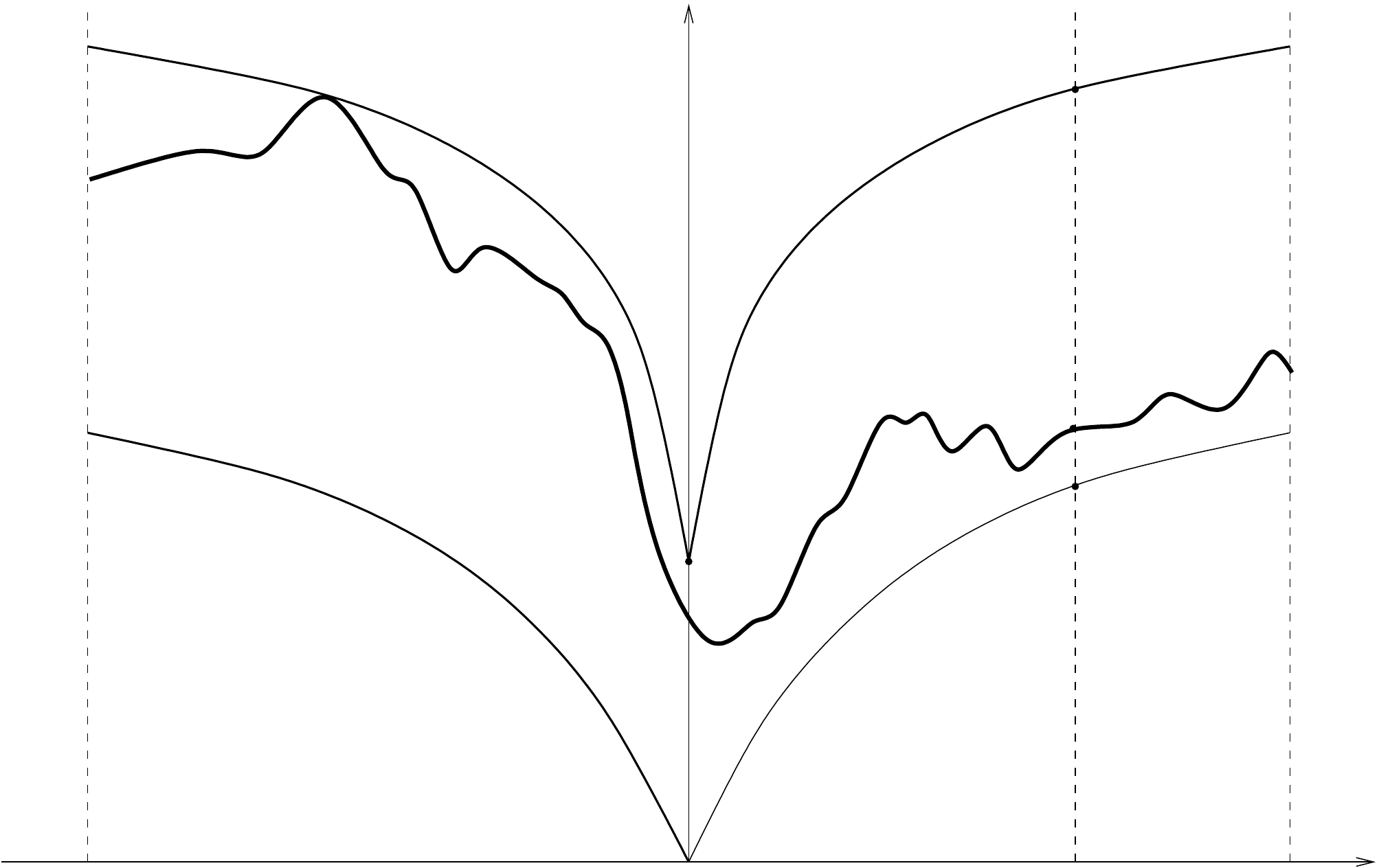}
\put(50.5,21){\scalebox{0.8}{$\xi^T_\varphi(Y)$}}
\put(62,57.5){\scalebox{.8}{$\xi^T_\varphi(Y)+\varphi(t)\,$}}
\put(72,34){\scalebox{0.8}{$Y(t)\,$}}
\put(78.8,25){\scalebox{0.9}{$\varphi_\lambda(t)$}}
\put(77.5,-3){\scalebox{1}{$t$}}
\put(92.5,-3){\scalebox{1}{$T$}}
\put(2.4,-3){\scalebox{1}{$-T$}}
\end{overpic}
\label{fig:b}
\caption{ Path $Y$ and $\xi^T_{\varphi} (Y )$ under $\bbE_T  \left[    \cdot  \big|\, \varphi_\lambda ,a\right]$.  }\vspace{-0.25cm}
\end{figure}

Consider 
\[
m_{n} =  m_n (a , \lambda ) \df \sup_{T} \bbE_{n, T} \left[ \xi^T_{\varphi} ( X_1 )| a, \lambda \right]. 
\]
If $m_1 <\infty$, then 
\eqref{eq:OneStepBound-1} implies that 
\[
 m_{n+1} \leq \frac{1}{\lambda^{1/3}} m_n + \sup_T \bbE_T \left[   \xi^T_{\varphi} (Y ) \big|\, \varphi_\lambda ,a \right] .
\]
Hence \eqref{eq:max-control} follows as soon as we shall check (see Figure \ref{fig:b})  that 
\be{eq:key-phi} 
m_1 (a, \lambda )  \leq \sup_T 
\bbE_T \left[   \xi^T_{\varphi} (Y ) \big|\, \varphi_\lambda ,a \right] <\infty .
\ee
Since $\xi^T_\varphi(\cdot) $ is monotone increasing, the first inequality in \eqref{eq:key-phi} follows by 
stochastic domination 
\eqref{eq:FKG-1}.  
The key point is to prove the second 
uniform bound \eqref{eq:key-phi}. 

In the sequel we shall assume that $a=1$ and, accordingly, 
shall drop it from all the notation. For instance, 
$\bbE_T \left[   \cdot  \big|\, \varphi_\lambda ,a \right]$ becomes 
$\bbE_T \left[   \cdot  \big|\, \varphi_\lambda  \right]$, and the corresponding partition 
function is recorded as ${\mathcal Z}_T(\varphi_\lambda )$. Moreover, we drop the superscript $T$ and write
$\xi_{\varphi} $ for $\xi^T_{\varphi} $. 

\subsection{Straightening of the boundary and Girsanov transform} 
The idea to use Girsanov's transform appeared in \cite{pascalzeitouni}. 
The floor should be smooth, and the singularity of $\varphi_\lambda$ at zero 
is a nuisance. 
However, by \eqref{eq:FKG-1} the bound 
\be{eq:key-phi-1} 
\sup_T 
\bbE_T \left[   \xi_{\varphi} (Y ) \big|\, h \right] <\infty .
\ee
with any $h\succ\varphi_\lambda$ implies \eqref{eq:key-phi}. 
We shall take a smooth symmetric $h= h_\lambda \succ \varphi_\lambda$ 
in such a way that: 
\be{eq:h-choice} 
\text{{$\varphi - h_\lambda$ is monotone on $\bbR_+$, }$h_\lambda = \varphi_\lambda$ outside some $[-T_0 , T_0]$ and $\max_t  h^{\prime\prime}_\lambda  (t ) \leq \frac{1}{2}$}. 
\ee
Recall that 
\begin{equation}
\label{eq:E-B-decomp}
 \begin{split} 
  \bbP_T  \left[  \cdot \big|\, h \right]
&= \frac{1}{{\mathcal Z}_T(h  )} 
\int_{h (-T )}^\infty
\int_{h (T )}^\infty 
{\mathbf B}_T^{x, y} 
\lb \, \cdot\, ;{\rm e}^{-\calA_T (X )}\1_{ h \prec X}\rb \dd y\dd x \\ 
&= 
\frac{1}{{\mathcal Z}_T(h  )}
\int_{h (-T)}^\infty 
{\mathbf P}_T^{x} 
\lb \, \cdot\, ;{\rm e}^{-\calA_T (X )}\1_{h \prec X}\rb \dd x , 
 \end{split}
\end{equation}
where ${\mathbf P}_T^{x}$ is the law of Brownian motion $X$ on $[-T,T]$ which starts at $x$ at time $-T$. 

We are going to derive a representation of 
${\mathcal Z}_T(h  )$ and, accordingly, of $\bbP_T\left[\cdot |h\right]$ in terms of polymers over flat 
wall, but with different tilts and boundary conditions. 
It, therefore, makes sense to  stress full  names for boundary conditions, 
floors and tilts. So, according to the notation introduced in 
\blue{\eqref{eq:mu-meas-bc}}, 
we are going to derive a representation of the quantities
${\mathcal Z}_T^{0, 0} (h, 1  )$  and $\bbP_T^{0,0}\left[\cdot |h , 1\right]$ corresponding to empty boundary conditions. 

Define $U (t ) = Y (t ) - h (t )$ and $u = x - h (-T )$. 
Thus $U$ satisfies the SDE  
\[ 
\dd U (t ) = \dd Y (t ) -
h^\prime (t )\dd t, \quad U (-T) = u.
\]
By Girsanov (used in the second equality below), 
\be{eq:Girsanov} 
\begin{split} 
 &{\rm e}^{\calA_{T} ( h  ) }
 {\mathbf P}_{T}^x \lb {\rm e}^{-\calA_{T} (X )}\1_{X\succ h }  \rb  = 
{\mathbf P}_{T}^u \lb {\rm e}^{-\calA_{T} (U )}\1_{\Omega_{+}^{T}} (U) \rb
\\
&= 
{\mathbf P}_{T}^u \lb 
{\rm e}^{-\int_{-T}^T h^\prime (t )\dd X (t ) - 
\frac{1}{2} \int_{-T}^T \lb h^\prime (t )\rb^2 \dd t - 
\calA_{T} ( X) } \1_{\Omega^{+}_{T}} (X ) \rb . 
\end{split} 
\ee
In the last line above $X$ is a ${\mathbf P}_{T}^u $-Brownian motion.
Under ${\mathbf P}_{T}^u$, 
\[ 
 \int_{-T}^T h^\prime (t )\dd X (t ) = 
 X h^\prime\Big|_{-T}^T - \int_{-T}^T X (t) h^{\prime\prime} (t )
 \dd t .
\]
Putting things together we conclude: Set 
\be{eq:nu-T}
\nu_T (u ) = h^\prime (T ) u = -h^\prime (-T ) u .
\ee
Then,  
\be{eq:Papprox} 
{\rm e}^{\calA_{T} \lb h + \frac{1}{2}(h^{\prime})^2 \rb  }
{\mathcal Z}_T^{0, 0} (h, 1  ) = 
{\mathcal Z}_T^{\nu_T, \nu_T} (0, 1 - h^{\prime\prime} ) .
%
\ee
A completely similar computation implies that the distribution of 
$Y$ under $\bbP_T^{0,0}[\cdot | h, 1]$ can be represented as the distribution  of $X + h = X +h_\lambda$ where $X$ has distribution 
$\bbP_T^{\nu_T,\nu_T}[\cdot | 0, 1 - h^{\prime\prime}]$. Since for $Y = X+h_\lambda$, 
\be{eq:xi-phi-h} 
\xi_{\varphi} (Y ) = \inf\lbr \xi\geq 0~:~ Y\prec \xi + \varphi\rbr = \xi_{\varphi - h_\lambda} (X ) \df \xi_{\psi_\lambda} (X ), 
\ee
we conclude that one can rewrite the expression in 
our 
target \eqref{eq:key-phi-1} as $$\sup_T 
\bbE_T^{\nu_T , \nu_T} \left[   \xi_{\psi_\lambda} (Y ) \big| 0, 1-h^{\prime\prime}  \right].$$ 
Since, by construction, $1-h^{\prime\prime} \geq \frac{1}{2}$, the stochastic domination \eqref{eq:FKG-1} 
enables a  reduction to: 
\be{eq:key-phi-2} 
\sup_T 
\bbE_T^{\nu_T , \nu_T} \left[   \xi_{\psi_\lambda} (Y ) \big| 0, \frac{1}{2}  \right] <\infty .
\ee
Furthermore, since by construction, $h^\prime (T ) > 0$ for all $T$ large enough, 
the second part of Lemma~\ref{lem:SD} implies that \eqref{eq:key-phi-2} will follow from
\be{eq:key-phi-3} 
\sup_T 
\bbE_T \left[   \xi_{\psi_\lambda} (Y ) \big| 0, \frac{1}{2}  \right] <\infty,
\ee
which corresponds to empty boundary conditions.
For the rest we shall focus on proving \eqref{eq:key-phi-3}.
\noindent 
\subsection{Proof of \eqref{eq:key-phi-3}} 
We start with some useful estimates for the partition functions $$Z_{0,t}^{x,y} :=Z_{0,t}^{x,y}(0,\infty,1/2).$$ 
Here and for the rest of this proof, with slight abuse of notation we adopt the convention that if $a$ is omitted from the notation then it corresponds to the case $a=\frac12$. 

First of all note, that for any $a>0$ the heat  kernel $Z_{0,t}^{x , y}( 0, \infty , a )$ has the following expansion: Let 
$\kappa_0^a , \kappa_1^a, \dots$ be the normalized (Dirichlet) eigenfunctions of $\frac{\dd^2}{2\dd x^2} - ax$ on 
$\bbL_2 \lb \bbR_+ \rb$, and let $0 > -\lambda_0 >-\lambda_1 >\dots$ be the corresponding
eigenvalues. Of course $\lambda_\ell = \frac{a}{b}\omega_\ell$ and $\kappa_\ell^a \propto \mathrm{Ai}\lb bx - \omega_\ell\rb$, 
where  $0 > -\omega_0 >  -\omega_1 >  \cdots $ 
are  zeroes of Airy 
function $\mathrm{Ai}$, and $b = \sqrt[3]{2a}$. Then, 
see e.g. Problem~1 in Chapter~9 of \cite{CodLev}, 
$\lbr \kappa_\ell^a\rbr$ is a complete orthonormal system, and
  by Riesz-Fisher and the elementary spectral theory the (Dirichlet) heat kernel 
\begin{equation}
\label{part_func_0}
Z_{0,t}^{x,y} 
=\sum_{m=0}^\infty e^{-\lambda_m t}\kappa_m(x)\kappa_m(y), 
\end{equation}
where we used the shortcut $\kappa_\ell = \kappa_{\ell}^{a}$ for $a=1/2$. 
The eigenfunctions $\kappa_m(x)$ are uniformly bounded. Furthermore, it is known that
zeros of the Airy function decrease relatively fast: $\omega_k\sim ck^{2/3}$ as $k\to\infty$.
These facts allow one to conclude that, uniformly in $t\ge t_0>0$,
\begin{equation}
\label{part_func_1}
\max_{x,y}Z_{0,t}^{x,y}\le \max_m\|\kappa_m\|_\infty^2\sum_{m=0}^\infty e^{-\lambda_m t}
\le C(t_0)e^{-\lambda_0 t}.
\end{equation}
Note that if the path $X(s)$ starting at $x$ does not go below $x/2$ then the corresponding
area is greater than $xt/2$. Therefore,
$$
\int_0^\infty Z_{0,t}^{x,y}dy \le e^{-xt/4}+\mathbf{P}_{0,t}^x\left(\min_{s\le t}X(s)<x/2\right).
$$
Using standard bound for the tail of the normal distribution, we conclude that there exist 
$t_0>0$ and 
$\gamma=\gamma(t_0) >0$ such that 
\begin{equation}
\label{part_func_2}
\int_0^\infty Z_{0,t_0}^{x,y}dy \le e^{-\gamma x} 
\end{equation}
for any $x >0$.
Combining \eqref{part_func_1} and \eqref{part_func_2}, we obtain
\begin{align}
\label{part_func_3}
\nonumber
Z_{0,t}^{x,y}
&=\int_0^\infty\int_0^\infty Z_{0,t_0}^{x,u}Z_{0,t-2t_0}^{u,v}Z_{0,t_0}^{v,y}dudv\\
\nonumber
&\le C(t_0)e^{-\lambda_0(t-2t_0)}\int_0^\infty\int_0^\infty Z_{0,t_0}^{x,u}Z_{0,t_0}^{v,y}dudv\\
&\le C_1(t_0)e^{-\lambda_0 t}e^{-\gamma(x+y)},\quad t\ge 3t_0.
\end{align}
It follows from the representation \eqref{part_func_0} that
$$
\lim_{t\to\infty}e^{\lambda_0t}Z_{0,t}^{x,y}=\kappa_0(x)\kappa_0(y)
$$
uniformly on compact subsets of $(0,\infty)^2$. Combining this with \eqref{part_func_3}
one can easily obtain
\begin{equation}
\label{part_func_4}
\lim_{T\to\infty}e^{2\lambda_0T}\calZ_T=\left(\int_0^\infty\kappa_0(x)dx\right)^2.
\end{equation}
We now derive an upper bound for the tail of the random variable $\xi_{\psi}$, where  $\psi=\psi_\lambda=\varphi-h_\lambda$. 
 {By \eqref{eq:h-choice} and \eqref{eq:varphi}  $\psi$ is a  symmetric function, it is monotone on $\bbR_+$, and it equals to $\frac{\lambda^\beta - 1}{\lambda^{\beta}   }\varphi $
 	outside $[-T_0 , T_0]$.}
 
Due to the symmetry of the function $\psi$,
$$
\bbP_T\left(\xi_{\psi}(Y)>r\right)
\le 2\bbP_T\left(Y(t)>\psi(t)+r\text{ for some }t\in[0,T]\right).
$$
There is no loss of generality to assume that $T\in \bbN$. 

Splitting $[0,T]$ into intervals of unit length and using the monotonicity of $\psi$,
we get
\begin{equation}
\label{tail_1}
\bbP_T\left(\xi_{\psi}(Y)>r\right)
\le 2 \sum_{k=0}^{T-1}\bbP_T\left(\max_{t\in[k,k+1]}Y(t)>\psi(k)+r\right).
\end{equation}
For every $k<T-1$ one has 
\begin{equation}
\label{tail_2}
\bbP_T\left(\max_{t\in[k,k+1]}Y(t)>\psi(k)+r\right)
=\frac{1}{\calZ_T}\int_0^\infty\int_0^\infty \calZ_{0,T+k}^{0,\delta_x}Q^{x,y}_k(r)\calZ_{0,T-k-1}^{\delta_y,0}dxdy,
\end{equation}
where
$$
Q^{x,y}_k(r)=\mathbf{B}^{x,y}_{k,k+1}\left(e^{-\frac{1}{2}\calA_{k,k+1}(Y)};\max_{t\in[k,k+1]}Y(t)>\psi(k)+r ; 
\1_{\Omega^+} (Y)
\right).
$$
It is immediate from \eqref{part_func_3} that
\begin{equation}
\label{tail_2a}
\calZ_{0,T+k}^{0,\delta_x}\le Ce^{-\lambda_0(T+k)}e^{-\gamma x}
\quad\text{and}\quad
\calZ_{0,T-k-1}^{\delta_y,0}\le Ce^{-\lambda_0(T-k-1)}e^{-\gamma y}.
\end{equation}
Applying these estimates and \eqref{part_func_4} to the corresponding terms in \eqref{tail_2},
we obtain
\begin{equation}
\label{tail_3}
\bbP_T\left(\max_{t\in[k,k+1]}Y(t)>\psi(k)+r\right)
\le C\int_0^\infty\int_0^\infty Q^{x,y}_k(r)e^{-\gamma(x+y)}dxdy,\ \;\;k<T-1.
\end{equation}
Set $\sfg (t ) = (2\pi)^{-\frac12}{\rm e}^{-t^2/2}$ and recall  that $\mathbf{B}^{x,y}_{0,1}(1) = \sfg (y-x )$.
By symmetry, we may assume without loss of generality that $x\leq y$. By the reflection principle for the Brownian bridge,
\begin{align*}
Q^{x,y}_k(r)&\leq \mathbf{B}^{x,y}_{k,k+1}\left(\max_{t\in[k,k+1]}Y(t)>\psi(k)+r\right)
\\&=
\left\{
\begin{array}{ll}
\sfg (y-x ),\ &y>\psi(k)+r,\\
\sfg (y- 2\psi(k)-2r+ x ),\ &y\le \psi(k)+r.
\end{array}
\right.
\end{align*}
Therefore,
\begin{align*}
\int_0^\infty Q^{x,y}_k(r)\dd r&\le \int_0^{(y-\psi(k))^+} 
\sfg (y-x )
\dd r
+\int_{(y-\psi(k))^+}^\infty
\sfg (y- 2\psi(k)-2r+ x )
\dd r\\
&=(y-\psi(k))^+
\sfg (y-x )
+\frac{1}{2}\int_{2\psi(k)+2(y-\psi(k))^+-x-y}^\infty
\sfg (z) 
\dd z\\
&\le (y-\psi(k))^++\frac{1}{2}{\rm 1}\{
y>\psi(k)/2\}+C\,e^{-\psi^2(k)/2}.
\end{align*}
Combining this with \eqref{tail_3}, we conclude that
$$
\int_0^\infty\bbP_T\left(\max_{t\in[k,k+1]}Y(t)>\psi(k)+r\right)dr
\le Ce^{-\gamma\psi(k)/2},\ \;\;\;k<T-1.
$$
Since $\psi(x)$ grows sufficiently fast, we conclude that, uniformly in $T$,
\begin{equation}
\label{tail_4}
\sum_{k=0}^{T-2}\int_0^\infty\bbP_T\left(\max_{t\in[k,k+1]}Y(t)>\psi(k)+r\right)dr\le C.
\end{equation}
For $k=T-1$ one has
$$
\bbP_T\left(\max_{t\in[T-1,T]}Y(t)>\psi(T-1)+r\right)
=\frac{1}{\calZ_T}\int_0^\infty
\int_0^\infty \calZ_{0,2T-1}^{0,\delta_x}Q^{x,y}_{T-1}(r)dxdy.
$$
Using \eqref{part_func_4} and \eqref{tail_2a}, we get
\begin{equation}
\label{tail_5}
\bbP_T\left(\max_{t\in[T-1,T]}Y(t)>\psi(T-1)+r\right)
\le C \int_0^\infty\int_0^\infty e^{-\gamma x}Q^{x,y}_{T-1}(r)dxdy.
\end{equation}
We infer from the definition of $Q^{x,y}_{T-1}(r)$ that
$$
\int_0^\infty Q^{x,y}_{T-1}(r)dy
\le\mathbf{P}^x_{T-1,T}\left(\max_{t\in[T-1,T]}Y(t)>\psi(T-1)+r\right).
$$
Integarting now over $r$, we obtain
\begin{align*}
&\int_0^\infty\int_0^\infty Q^{x,y}_{T-1}(r)dydr
\le\int_0^\infty\mathbf{P}^0_{0,1}\left(\max_{t\in[0,1]}Y(t)>\psi(T-1)-x+r\right)dr\\
&\hspace{2cm}\le (x-\psi(T-1))^++2\int_{(\psi(T-1)-x)^+}^\infty\mathbf{P}^0_{0,1}\left(Y(1)>r\right)dr\\
&\hspace{2cm}\le (x-\psi(T-1))^++\sfg(\psi(T-1)-x)^+).
\end{align*}
Integrating \eqref{tail_5} over $r$ and applying the latter bound, we arrive at 
\begin{equation}
\label{tail_6}
\int_0^\infty \bbP_T\left(\max_{t\in[T-1,T]}Y(t)>\psi(T-1)+r\right)dr
\le C e^{-\gamma\psi(T-1)/2}.
\end{equation}
It remains to note that \eqref{eq:key-phi-3} is immediate from \eqref{tail_1}, \eqref{tail_4} and \eqref{tail_6}.

\appendix
\section{Polymer measures are well defined} 
\label{sub:Exists} 
We shall describe conditions under which 
probability measures in \eqref{eq:mu-meas} are well defined, or, equivalently, 
under which the corresponding 
partition functions 
$Z_{n, T}^{\ux , \uy}( h_- , h_+, \urho )$ are finite. 
Define the minimal tilt on the interval $[-T , T]$ as 
\be{eq:def-a} 
a = a_T = \min_k \min_{t\in [-T,T]} \rho_k (t ) >0 .
\ee
and let $\underline{a}$ be the corresponding tuple of constant functions. Evidently, 
\[ 
Z_{n, T}^{\ux , \uy}( h_- , h_+, \urho ) \leq Z_{n, T}^{\ux , \uy}( 0 , \infty , \ua )
\]
Define $\uu  =\lbr x_1, \dots , x_{n-1}\rbr$ and $\uv = \lbr y_1, \dots , y_{n-1}\rbr$ and assume
that $Z_{n-1, T}^{\uu , \uv}( 0, \infty , \ua )<\infty$. Then,
\be{eq:Z-reduction}
\begin{split}
	Z_{n, T}^{\ux , \uy}( 0 , \infty , \ua ) &=  Z_{n-1, T}^{\uu , \uv}( 0, \infty , \ua )\cdot \bbE_{n-1, T}^{\uu , \uv}\left[ 
	{\mathbf B}_T^{x_n, y_n}\lb {\rm e}^{-a\calA_T (Y )}\1_{0\prec Y \prec X_{n-1}}\rb~\big|\, 0 , \infty , \ua\right]  \\
	&\leq Z_{n-1, T}^{\uu , \uv}( 0, \infty , \ua )Z_{T}^{x_n , y_n}( 0, \infty , a )\leq\dots 
	\leq \prod_1^n Z_{T}^{x_k , y_k}( 0, \infty , a ),  
\end{split}
\ee
where the first inequality above folows by removing the constraint $Y\prec X_{n-1}$. 

Consequently, general partition functions in  \eqref{eq:mu-meas-bc} may be bounded above as
\be{eq:pb-nu-bound} 
\calZ_{n, T}^{\unu , \ueta}( h_- , h_+, \urho ) \leq \prod_{k=1}^n \lb \int_0^\infty\int_0^\infty {\rm e}^{-\nu_k (x )}
Z_{T}^{x , y}( 0, \infty , a ){\rm e}^{-\eta_k (y )}\dd x\dd y\rb . 
\ee
As it was already briefly explained in the paragraph preceding \eqref{part_func_0}, 
 the kernel $Z_{T}^{x , y}( 0, \infty , a )$ has the following expansion: \be{eq:Zt-expAiry} 
Z_{T}^{x , y}( 0, \infty , a ) = \sum_0^\infty {\rm e}^{-2\lambda_\ell T} \kappa_\ell^a (x )\kappa_\ell^a (y ) .
\ee
We, therefore, conclude: 
\begin{lemma} 
	\label{lem:Exists} Consider \eqref{eq:def-a}. Let 
	$\kappa_0 , \kappa_1, \dots$ be the normalized eigenfunctions of $\frac{\dd^2}{2\dd x^2} - ax$ on 
	$\bbL_2 \lb \bbR_+ \rb$, and let $0 > -\lambda_0   >-\lambda_1   >\dots$ be the corresponding
	eigenvalues.
	Assume that 
	\[
	\sum_{\ell=0}^\infty {\rm e}^{-2\lambda_\ell T} \lb \int   {\rm e}^{-\nu_k (x )}\kappa_\ell (x )\dd x\rb 
	\lb \int   {\rm e}^{-\eta_k (y )}\kappa_\ell (y )\dd y\rb < \infty  
	\]
	is absolutely convergent for every $k=1, \dots, n$. 
	Then, $\bbP_{n, T}^{\unu , \ueta}\left[\, \cdot\, \big|  h_- , h_+, \urho \right]$
	in \eqref{eq:mu-meas-bc} is well defined. In particular, it is well defined whenever ${\rm e}^{-\nu_k}$-s and ${\rm e}^{-\eta_k}$-s 
	belong to $\bbL_2 (\bbR_+ )$, and, since $\lim_{y\to -\infty} \int_y^\infty \mathrm{Ai} ( x )\dd x$ exists and finite, the
	measures 
	$\bbP_{n, T}\left[\, \cdot\, \big|  h_- , h_+, \urho \right]$
	are  well defined also in the case of empty 
	boundary conditions $\unu, \ueta = \underline{0}$.
\end{lemma}

\section{Proof of the stochastic domination lemma} 
\label{sub:SDproof} 
\begin{proof}[Proof of Lemma \ref{lem:SD}]
As in the proof of Lemma~2.6 
in \cite{corwinhammond} we construct a coupling for discrete random walk ensembles via Markov chains and then obtain the desired result by appealing to the invariance principle. The presence of area tilts and boundary conditions makes our setting slightly different from that of \cite{corwinhammond}. For completeness we provide the details below. 

We start with the case of fixed boundary conditions \eqref{eq:FKG-2}.  
For each $N,n\in\bbN$, let $T_N=\lfloor TN\rfloor$, and consider vectors of integers $\ux_N=(x_{N,1},\dots,x_{N,n})$ and $\uy_N=(y_{N,1},\dots,y_{N,n})$, such that, as $N\to\infty$, $\frac1{\sqrt N}x_{N,i}=(1+o(1))x_i$, $\frac1{\sqrt N}y_{N,i}=(1+o(1))y_i$, for $i=1,\dots,n$.  Given the floor and ceiling functions $h_\pm$, consider height functions $h_{N,\pm}$ such that $\frac1{\sqrt N}h_{N,\pm}(k)= (1+o(1))h_{\pm}(k/N)$, uniformly in $k\in\{-T_N,\dots,T_N\}$. Let 
$\Omega(N,n)$  denote the set of vectors $\underline W=(W_1,\dots,W_n)$ where $W_i$
denotes a lattice path $\{W_i(k)\in\bbZ,\,k=-T_N,\dots,T_N\}$ satisfying $|W_i(k+1)-W(k)|=1$ for all $k\in\{-T_N,\dots,T_N-1\}$ and such that $0\leq W_{i+1}(k)<W_i(k)$ for all $i,k$.
Finally, let $\bbP_{N,n, T}^{\ux, \uy}\left[\cdot   | h_{-} , h_{+}, \urho\right]$ denote the probability measure 
on $\Omega(N,n)$ associated to the partition function
$$
\sum_{\underline W\in\Omega(N,n)} \1_{\underline W(-T_N)=\ux_N,\, \underline W(T_N)=\uy_N}\1_{W_1\leq h_{N,+}}\1_{ h_{N,-}\leq W_n}\,{\rm e}^{-\frac1{N^{3/2}}\sum_1^n \calA_{\rho_i,N} (W_i )},
$$
where $$\calA_{\rho_i,N} (W_i ) = \sum_{k=-T_N}^{T_N}\rho_i(k/N)W_i(k).$$
Next, define the rescaled paths $\underline{\hat W}_N(t) = \frac1{\sqrt N} \underline W(tN)$, $t\in[-T,T]$,
where the value of $\underline W$ at non-integer points is defined by linear interpolation. 
Call $\hat \bbP_{N,n, T}^{\ux, \uy}\left[\cdot   | h_{-} , h_{+}, \urho\right]$ the law of the continuous paths $\underline{\hat W}_N$ induced by $\bbP_{N,n, T}^{\ux, \uy}\left[\cdot   | h_{-} , h_{+}, \urho\right]$. Since for all $i$, the Riemann sum
$$
\frac1{N^{3/2}}\calA_{\rho_i,N} (W_i ) = \frac1N \sum_{k=-T_N}^{T_N}\rho_i(k/N)\hat W_{N,i}(k/N)
$$
approximates the integral $\int_{-T}^T\rho_i(t)\hat W_{N,i}(t) dt$, 
the invariance principle implies that for all fixed $n,T$, the probability measures $\hat \bbP_{N,n, T}^{\ux, \uy}\left[\cdot   | h_{-} , h_{+}, \urho\right]$ converge weakly as $N\to\infty$ to the probability measure $\bbP_{n, T}^{\ux, \uy}\left[\cdot   | h_- , h_+, \urho\right]$.

The same construction can be repeated for the measure $\bbP_{n, T}^{\uu , \uv}\left[\cdot  | g_- , g_+, \ukappa \right]$. It is not hard to check that, under the current assumptions, the sequences $\ux_N,\uy_N,\uu_N,\uv_N$ and $h_{N,\pm},g_{N,\pm}$ associated to the given boundary data can be chosen in such a way that, for all $N$ large enough:
\begin{itemize}
\item $h_{N,\pm}(k)\leq   g_{N,\pm}(k)$ for all $k=-T_N,\dots,T_N$;
\item for every $i$, $x_{N,i}\leq   u_{N,i}$, and $y_{N,i}\leq   v_{N,i}$; 
\item for every $i$, $x_{N,i}$ and $u_{N,i}$ are integers with the same parity, and the same applies to $y_{N,i},v_{N,i}$;
\item the set of $\underline W\in\Omega(N,n)$ satisfying the boundary constraints $$\underline W(-T_N)=\ux_N,\, \underline W(T_N)=\uy_N, \;\;W_1\leq h_{N,+}, \,h_{N,-}\leq W_n,$$ is not empty, and the same applies with $\uu_N,\uv_N$ and $g_{N,\pm}$.
\end{itemize}
Then, the desired statement 
$$\bbP_{n, T}^{\ux, \uy}\left[\cdot   | h_- , h_+, \urho\right]\fkg 
\bbP_{n, T}^{\uu , \uv}\left[\cdot  | g_- , g_+, \ukappa \right],$$
follows if, for all large enough $N$,  we can construct a coupling $(\underline W,\underline W')$ on $\Omega(N,n)\times \Omega(N,n)$ of the probability measures $\bbP_{N,n, T}^{\ux, \uy}\left[\cdot   | h_- , h_+, \urho\right]$ and $\bbP_{N,n, T}^{\uu , \uv}\left[\cdot  | g_- , g_+, \ukappa \right]$  such that with probability one  for each $i=1,\dots,n$, $k=-T_N,\dots,T_N$ one has $W_i(k)\leq W'_i(k)$. 

The coupling $(\underline W,\underline W')$ is defined as a limit of Markov chain couplings. 
We consider the heat bath chain for the discrete polymer ensemble. This is the discrete time Markov chain on $\Omega(N,n)$ such that at each time step a vertex $k\in\{-T_N+1,\dots,T_N-1\}$, an index $i\in\{1,\dots,n\}$, and a real number $U\in[0,1]$ are picked independently and uniformly at random; if $W_i(k-1)\neq W_i(k+1)$ then nothing happens; if   $W_i(k-1)= W_i(k+1)$, then 
$W_i(k)$ is replaced by $W_i(k-1) + 1$ if $U\leq p_{k,i}$ and by $W_i(k-1) - 1$ if $U> p_{k,i}$, where we use the notation
$$
p_{k,i}= \frac{{\rm e}^{-2\rho_i(k/N)N^{-3/2}}}{1+{\rm e}^{-2\rho_i(k/N)N^{-3/2}}};
$$
if the new polymer configuration $\underline W$ violates either of the constraints $\underline W\in\Omega(N,n)$, $W_1\leq h_{N,+}$, $h_{N,-}\leq W_n$, then the proposed update is rejected; otherwise the current configuration is updated accordingly.  The above defined Markov chain is reversible with respect to the measure $\bbP_{N,n, T}^{\ux, \uy}\left[\cdot   | h_- , h_+, \urho\right]$, and converges to it as time goes to infinity, for any valid initial condition. Now, suppose that $\underline W,\underline W'$ are two polymer configurations in $\Omega(N,n)$ such that $W_1\leq h_{N,+}$, $h_{N,-}\leq W_n$ and $W'_1\leq g_{N,+}$, $g_{N,-}\leq W'_n$, and suppose further that $W_i(k) \leq W_i'(k)$
at every $i,k$. A coupling of the single Markov chain step for this pair is obtained by repeating the above described updating procedure with the same choice of random numbers $k,i,U$
for both copies. Since $h_{N,\pm}\leq g_{N,\pm}$ and $\rho_i\geq \kappa_i$ it follows that the new polymer configurations must satisfy again  $W_i(k) \leq W_i'(k)$. Indeed, because of the parity assumption on the boundary heights, the first violation of this condition could only appear at a site $k$ such that $$W_i(k-1) = W_i(k+1)=W'_i(k-1) = W_i'(k+1),$$ and in this case the conditions $h_{N,\pm}\leq g_{N,\pm}$ and $\rho_i\geq \kappa_i$ guarantee that the order is preserved. 
Repeating this procedure at each time step yields a Markov chain coupling such  that if at time zero
one has $W_i(k) \leq W_i'(k)$ at every $i,k$, then this condition is preserved at all times. The initial polymer configurations can be chosen by taking $W$ as the minimal element of $\Omega(N,n)$ such that $\underline W(-T_N)=\ux_N,\, \underline W(T_N)=\uy_N$, $W_1\leq h_{N,+}, h_{N,-}\leq W_n$, and $W'$ as the maximal element of $\Omega(N,n)$ such that $\underline W(-T_N)=\uu_N,\, \underline W(T_N)=\uv_N$, $W_1\leq g_{N,+}, g_{N,-}\leq W_n$.   As pointed out above, these initial configurations are well defined. 
It follows that at time zero, and thus at all times,   $W_i(k) \leq W_i'(k)$ at every $i,k$. By taking time to infinity one obtains the desired coupling of $\bbP_{N,n, T}^{\ux, \uy}\left[\cdot   | h_- , h_+, \urho\right]$ and $\bbP_{N,n, T}^{\uu , \uv}\left[\cdot  | g_- , g_+, \ukappa \right]$. This ends the proof of \eqref{eq:FKG-2}.

To prove the  statement \eqref{eq:FKG-1}, we need to take into account  the boundary conditions encoded by the functions $\uxi,\uzeta,\unu,\ueta$. 
Let $\bbP_{N,n, T}^{\uxi , \uzeta}\left[\cdot   | h_{-} , h_{+}, \urho\right]$ denote the probability measure 
on $\Omega(N,n)$ associated to the partition function
\begin{align*}
&\sum_{\underline\ell,\underline r}\sum_{W\in\Omega(N,n)} \1_{\underline W(-T_N)=\uell,\, \underline W(T_N)=\underline r}\1_{W_1\leq h_{N,+}}\1_{ h_{N,-}\leq W_n}\,\times\\
&\qquad\times \, {\rm e}^{-\sum_1^n \xi_i(\ell_i/\sqrt N)}{\rm e}^{-\sum_1^n \zeta_i(r_i/\sqrt N)}\,{\rm e}^{-\frac1{N^{3/2}}\sum_1^n \calA_{\rho_i,N} (W_i )},
\end{align*}
where $\underline\ell,\underline r$ range over all vectors $(k_1,\dots,k_n)\in\bbZ^n$ such that $0\leq k_n<k_{n-1}<\cdots<k_1$,  and $h_{N,\pm}$ is such that $\frac1{\sqrt N}h_{N,\pm}(k)= (1+o(1))h_{\pm}(k/N)$, uniformly in $k\in\{-T_N,\dots,T_N\}$. As above we call $\hat \bbP_{N,n, T}^{\uxi, \uzeta}\left[\cdot   | h_{-} , h_{+}, \urho\right]$ the law induced on the rescaled continuous paths $\underline {\hat W}_N$. 
Then, approximating the sum over $\uell,\ur$ by integrals,  the invariance principle implies that for all fixed $n,T$, the probability measures $\hat \bbP_{N,n, T}^{\uxi, \uzeta}\left[\cdot   | h_{-} , h_{+}, \urho\right]$ converge weakly as $N\to\infty$ to the probability measure $\bbP_{n, T}^{\uxi, \uzeta}\left[\cdot   | h_- , h_+, \urho\right]$.
Therefore, the desired statement 
\begin{equation}\label{eq:stdome}
\bbP_{n, T}^{\uxi , \uzeta}\left[\cdot   | h_- , h_+, \urho\right]\fkg 
\bbP_{n, T}^{\unu , \ueta}\left[\cdot  | g_- , g_+, \ukappa \right],
\end{equation}
follows if, for all $N$ large enough, we can construct a coupling $(\underline W,\underline W')$ on $\Omega(N,n)\times \Omega(N,n)$ of the probability measures $\bbP_{N,n, T}^{\uxi, \uzeta}\left[\cdot   | h_- , h_+, \urho\right]$ and $\bbP_{N,n, T}^{\unu , \ueta}\left[\cdot  | g_- , g_+, \ukappa \right]$  such that with probability one for each $i=1,\dots,n$, and $k=-T_N,\dots,T_N$, one has $W_i(k)\leq W'_i(k)$.


The coupling $(\underline W,\underline W')$ is defined as before with the only difference that the random index $k$ is now picked uniformly in $\{-T_N,\dots,T_N\}$. If $k\notin\{-T_N,T_N\}$ then we repeat the previously described update rule. If $k=-T_N$, then the height $W_i(-T_N)$ is replaced by $ W_i(-T_N+1)+1$ if $U\leq \hat p_{-T_N,i} (W_i(-T_N+1) )$ and by $ W_i(-T_N+1)-1$ if $U> \hat p_{-T_N,i} (W_i(-T_N+1) )$, where 
$$
\hat p_{-T_N,i} (a ) := 
 \frac{{\rm e}^{-2\rho_i(-T_N/N)N^{-3/2}}\,{\rm e}^{-\nabla\xi_i(
		a
 		)}}{1+{\rm e}^{-2\rho_i(-T_N/N)N^{-3/2}}\,{\rm e}^{-\nabla\xi_i(
 		a 
 		)}},
$$ 
and we use the notation $$\nabla\xi_i(a) = \xi_i\left(\tfrac{a+1}{\sqrt N}\right) - \xi_i\left(\tfrac{a-1}{\sqrt N}\right).$$
Similarly, 
$W_i^{\prime}(-T_N)$ is replaced by $ W_i^{\prime}(-T_N+1)+1$ if $U\leq \hat {q}_{-T_N,i} \lb W_i^{\prime}(-T_N+1)\rb  $ and by $ W_i^{\prime}(-T_N+1)-1$ if $U> \hat {q}_{-T_N,i} \lb W_i^{\prime}(-T_N+1)\rb $, where 
$$
\hat {q}_{-T_N,i} (
a)  = \frac{{\rm e}^{-2\kappa_i(T_N/N)N^{-3/2}}\,{\rm e}^{-\nabla\nu_i(
		a)}}{1+{\rm e}^{-2\kappa_i(T_N/N)N^{-3/2}}\,{\rm e}^{-\nabla
		\nu_i(a) }},
$$ 
with 
$$\nabla\nu_i(a) = \nu_i\left(\tfrac{a+1}{\sqrt N}\right) - \nu_i\left(\tfrac{a-1}{\sqrt N}\right).$$

As before we may assume without loss of generality that the configurations $W_i,W_i'$ at time zero are such that $W_i(-T_N)$ and $W_i'(-T_N)$ have the same parity and that the same applies to $W_i(T_N)$ and $W_i'(T_N)$. Note also that the parity of these boundary values does not change with time. 
Thanks to this parity constraint, the first violation of the order $W_i\leq W'_i$  can occur at site $k=-T_N$ only if for some 
$i$ one has $W_i(-T_N+1)=W'_i(-T_N+1)$. Therefore, it is sufficient to show that for all $a\geq 0$:
\begin{align*}
&\nabla\xi_i(a) - \nabla\nu_i(a)\geq 0\,.
\end{align*}
The bound above follows immediately from the assumption $\xi_i' (x ) \geq \nu_i' (x )$, $x\geq 0$.
This implies that our coupling preserves the order at the boundary $-T_N$. The same argument applies to the boundary at  $T_N$. A repetition of the previous argument then concludes the proof of \eqref{eq:FKG-1}.
\end{proof}
\noindent{\bf Acknowledgment.} {P.C and D.I would like to thank Fabio Martinelli and Yvan Velenik for useful discussions at the initial stages of this project.}

\bibliographystyle{plain}

\bibliography{bib_tightness}
\end{document}